\documentclass[a4paper]{amsart}

\usepackage[backrefs,lite]{amsrefs}

\usepackage{mathtools}   
\usepackage{amssymb}              
\usepackage{lmodern,bm}              
\usepackage{longtable}
\usepackage{pgfplots}
\pgfplotsset{compat=1.18}
\usepackage{tikz}
\usepackage[all]{xy}
\usepackage{dsfont}
\usepackage{pdflscape}
\usepackage{bigdelim}
\usepackage{multirow}
\usepackage{cancel}
\usepackage{array}
\usepackage[hidelinks]{hyperref}
\usetikzlibrary{cd}

\usepackage{algorithm}
\usepackage[noEnd=false,indLines=false]{algpseudocodex}
\algrenewcommand\algorithmicrequire{\textbf{Input:}}

\CompileMatrices

\newtheorem{theorem}{Theorem}[section]

\newtheorem{lemma}[theorem]{Lemma}
\newtheorem{proposition}[theorem]{Proposition}

\theoremstyle{definition}

\newtheorem{definition}[theorem]{Definition}

\newtheorem{example}[theorem]{Example}
\newtheorem{remark}[theorem]{Remark}

\newtheorem{algrthm}[theorem]{Algorithm}

\numberwithin{equation}{theorem}

\def\vector2#1#2{\left(\begin{array}{c} #1 \\ #2 \end{array}\right)}

\def\Cl{{\rm Cl}}

\def\ZZ{{\mathbb Z}}
\def\RR{{\mathbb R}}

\def\QQ{{\mathbb Q}}

\def\Mat{{\rm Mat}}

\def\mult{{\rm mult}}

\def\conv{{\rm conv}}

\def\bangle#1{{\langle #1 \rangle}}

\def\PS#1#2{{\sum_{\nu=0}^{\infty} #1_{\nu} #2^{\nu}}}

\def\GL{{\rm GL}}

\def\GL{{\rm GL}}

\def\Hom{{\rm Hom}}

\def\mult{{\rm mult}}

\def\lcm{{\rm lcm}}

\def\lin{{\rm lin}}

\def\red{{\rm red}}
\def\vol{{\rm vol}}
\def\Vol{{\rm Vol}}
\def\NF{{\rm NF}}
\def\HNF{{\rm HNF}}
\def\SUM{{\rm SUM}}
\def\ORD{{\rm ORD}}
\def\PS{{\rm PS}}

\makeatletter
\newcommand*{\defeq}{\mathrel{\rlap{%
                     \raisebox{0.3ex}{$\m@th\cdot$}}%
                     \raisebox{-0.3ex}{$\m@th\cdot$}}%
                     =}
\makeatother

\title[Sharp volume and multiplicity bounds for Fano simplices]{Sharp volume and multiplicity bounds \\ for Fano simplices}

\author[Andreas Bäuerle]{Andreas Bäuerle}

\address{Mathematisches Institut, Universität Tübingen,
Auf der Morgenstelle 10, 72076 Tübingen, Germany}
\email{baeuerle@math.uni-tuebingen.de}


\begin{document}

\maketitle

\begin{abstract}
We present sharp upper bounds on the volume, Mahler volume and multiplicity for Fano simplices depending on the dimension and Gorenstein index. These bounds rely on the interplay between lattice simplices and unit fraction partitions. Moreover, we present an efficient procedure for explicitly classifying Fano simplicies of any dimension and Gorenstein index and we carry out the classification up to dimension four for various Gorenstein indices.
\end{abstract}


\section{Introduction}

\emph{Fano simplices} are lattice simplices $\Delta$ with primitive vertices containing the origin in their interior. Their geometric counterparts are \emph{fake weighted projective spaces} $Z$, ie.~$\QQ$-factorial toric Fano varieties of Picard number one, where Fano means normal and projective with ample anticanonical class~$-\mathcal{K}$. The \emph{Gorenstein index} of~$Z$ is the smallest positive integer $g$ such that $-g\mathcal{K}$ is Cartier. For the corresponding Fano simplex $\Delta$, the Gorenstein index equals the smallest positive integer $g$ such that the~$g$-fold of the dual polytope $\Delta^*$ is a lattice simplex.

Our first result provides sharp volume bounds. The \emph{normalized volume} of a~$d$-dimensional polytope is the $d!$-fold of its euclidean volume. 
For simplices of Gorenstein index one (also called \emph{reflexive} simplices), Nill \cite{Ni07}*{Thm.~A} provides sharp upper bounds on the normalized volume in terms of the dimension of the simplex. We extend these results to Fano simplices of arbitrary Gorenstein index.
We write~$\Delta = \Delta(P)$ with the $d \times (d+1)$ matrix $P$ having the vertices of $\Delta$ as its columns. For any integer $g \ge 1$ we define the \emph{$g$-Sylvester sequence}~$s_{g,1},s_{g,2},\dots$ and the \emph{truncated $g$-Sylvester sequence}~$t_{g,1},t_{g,2},\dots$ as
\begin{equation*}
    s_{g,1} \ := \ g+1,
    \qquad
    s_{g,k+1} \ := \ s_{g,k}(s_{g,k}-1)+1,
    \qquad
    t_{g,k} \ := \ s_{g,k}-1.
\end{equation*}

\begin{theorem}\label{thm:max-vol}
There are sharp upper bounds on the normalized volume of a Fano simplex $\Delta$, only depending on its dimension $d$ and its Gorenstein index $g$:
\goodbreak
\begin{enumerate}
    \item
    Assume $(d,g) = (2,1)$. We have the following upper bound on the normalized volume of $\Delta$, which is attained if and only if~$\Delta \cong \Delta(P)$ holds:
    \begin{equation*}
        \quad\qquad
        \Vol(\Delta) \ \le \ 9,
        \qquad\qquad
        P
        \ = \
        \left[\begin{array}{ccc}
            1 & 1 & -2 \\
            0 & 3 & -3
        \end{array}\right].
    \end{equation*}

    \item
    In all other cases the normalized volume of $\Delta$ is bounded from above by
    \begin{equation*}
        \Vol(\Delta) \ \le \ \frac{2t_{g,d}^2}{g^2}.
    \end{equation*}
    Equality holds if and only if we have~$\Delta \cong \Delta(P)$, where $P$ is one of the following:
    \begin{equation*}
        P
        \ = \
        \left[\begin{array}{rrrr}
            1 & 1 & 1 & -5\\
            0 & 2 & 2 & -4\\
            0 & 0 & 6 & -6
        \end{array}\right],
    \end{equation*}
    \begin{equation*}
        P
        \ = \
        \left[\begin{array}{rrrrrr}
            1      & 0      & \dots  & 0      & \frac{(s_{g,1}-g)}{s_{g,1}}\frac{t_{g,d}}{g} & -\frac{(s_{g,1}+g)}{s_{g,1}}\frac{t_{g,d}}{g} \\
            0      & 1      & \ddots & \vdots & \vdots     & \vdots \\
            \vdots & \ddots & \ddots & 0      & \vdots     & \vdots \\
            \vdots &        & \ddots & 1      & \frac{(s_{g,d-1}-g)}{s_{g,d-1}}\frac{t_{g,d}}{g} & -\frac{(s_{g,1}+g)}{s_{g,1}}\frac{t_{g,d}}{g} \\
            0      & \dots  &\dots   & 0      & \frac{t_{g,d}}{g}     & -\frac{t_{g,d}}{g} \\
        \end{array}\right].
    \end{equation*}
\end{enumerate}
\end{theorem}

Another invariant of a Fano simplex $\Delta$ is its \emph{multiplicity}, ie. the order of the sublattice generated by the vertices of $\Delta$. For Fano simplices having only the origin as an interior lattice point, for instance reflexive ones, ~\cite{AvKaLeNi21}*{Thm.~1.1} provides sharp upper bounds on the multiplicity in terms of the dimension.
In our second result we provide multiplicity bounds for arbitrary Fano simplices in terms of the Gorenstein index and the dimension.

\begin{theorem}\label{thm:fwps-max-mult}
There are upper bounds on the multiplicity of any Fano simplex $\Delta$, only depending on its dimension $d$ and its Gorenstein index $g$:
\begin{enumerate}
    \item
    Assume $d = 3$ and $g \in \{1,2\}$. We have the following upper bound on the multiplicity of $\Delta$, which is attained if and only if~$\Delta \cong \Delta(P)$ holds:
    \begin{equation*}
        \quad\qquad
        \mult(\Delta) \ \le \ 16g^2,
        \qquad\qquad
        P
        \ = \
        \left[\begin{array}{cccc}
            1 & 4g-3 & 4g-3 & 5-8g \\
            0 & 4g   & 0    & -4g  \\
            0 & 0    & 4g   & -4g
        \end{array}\right].
    \end{equation*}

    \item
    Assume $(d,g) = (4,1)$. We have the following upper bound on the multiplicity of $\Delta$, which is attained if and only if attained if~$\Delta \cong \Delta(P)$ holds:
    \begin{equation*}
        \quad\qquad
        \mult(\Delta) \ \le \ 128,
        \qquad\qquad
        P
        \ = \
        \left[\begin{array}{ccccc}
            1 & 1 & 1 & 1 & -7 \\
            0 & 2 & 2 & 2 & -6 \\
            0 & 0 & 8 & 0 & -8 \\
            0 & 0 & 0 & 8 & -8
        \end{array}\right].
    \end{equation*}
    
    \item
    In all other cases the multiplicity of $\Delta$ is bounded from above by
    \begin{equation*}
        \mult(\Delta) \ \le \ \frac{3t_{g,d-1}^2}{g}.
    \end{equation*}
    If equality holds, then we either have $(d,g) = (3,3)$ and $\Delta \cong \Delta(P)$ holds, where
    \begin{equation*}
        P
        \ = \
        \left[\begin{array}{cccc}
            1 & 1  & 5  & -7  \\
            0 & 12 & 0  & -12 \\
            0 & 0  & 12 & -12
        \end{array}\right],
    \end{equation*}
    or there are positive integers $a_{1}, \dots, a_{d-1} \in \ZZ_{\ge 1}$ such that $\Delta \cong \Delta(P)$ holds, where $P$ is the matrix:
    \begin{equation*}
        \quad\qquad
        \left[\setlength{\arraycolsep}{3pt}\begin{array}{rrrrrrr}
            1      & 0      & \dots  & 0      & \frac{(s_{g,1}-g)}{s_{g,1}}\frac{t_{g,d-1}}{g}         & a_{1}      & -\left( \frac{(s_{g,1}+2g)}{s_{g,1}}\frac{t_{g,d-1}}{g} + a_{1} \right)  \\
            0      & 1      & \ddots & \vdots & \vdots              & \vdots      & \vdots   \\
            \vdots & \ddots & \ddots & 0      & \vdots              & \vdots      & \vdots   \\
            \vdots &        & \ddots & 1      & \frac{(s_{g,d-2}-g)}{s_{g,d-2}}\frac{t_{g,d-1}}{g}      & a_{d-2}  & -\left( \frac{(s_{g,d-2}+2g)}{s_{g,d-2}}\frac{t_{g,d-1}}{g} + a_{d-2} \right) \\
            0      & \dots  &\dots   & 0      & \frac{t_{g,d-1}}{g} & a_{d-1}  & -\left( \frac{t_{g,d-1}}{g} + a_{d-1} \right) \\
            0      & \dots  &\dots   & 0      & 0                   & 3 t_{g,d-1} & -3 t_{g,d-1}   \\
        \end{array}\right]
    \end{equation*}
    Moreover, if $g$ is odd, then for $k = 1,\dots,d-2$ we may choose
    \begin{equation*}
        a_k \ = \ \frac{(s_{g,k}-g)}{s_{g,k}}\frac{t_{g,d-1}}{g},
        \qquad\qquad
        a_{d-1} \ = \ \frac{t_{g,d-1}}{g}.
    \end{equation*}
\end{enumerate}
\end{theorem}

For our third result we consider the \emph{Mahler volume} \cite{Ku08} of a (not necessarily Fano) rational \emph{IP simplex}. By an IP simplex we mean a rational simplex $\Delta$, that has the origin in its interior and its Mahler volume is the product of the normalized volume of $\Delta$ and the normalized volume of its dual polytope $\Delta^*$. Again, we obtain sharp upper bounds that only depend on the dimension and the Gorenstein index. For the Gorenstein index of a rational IP simplex see Definition \ref{def:gorind}.

\begin{theorem}\label{thm:max-mahler}
Let $\Delta$ a $d$-dimensional IP simplex of Gorenstein index $g$. Then we have
\begin{equation*}
    \Vol(\Delta) \Vol(\Delta^*) \ \le \ \frac{t_{g,d+1}^2}{g^{d+2}}.
\end{equation*}
Equality holds if and only if there is $H \in \GL(d,\QQ)$ such that $\Delta \cong H\cdot \Delta(P)$ holds, where
\begin{equation*}
    P
    \ = \
    \left[\begin{array}{ccccc}
             1 &      0 &  \dots & 0      & -\frac{t_{g,d+1}}{s_{g,1}} \\
             0 & \ddots & \ddots & \vdots & \vdots \\
        \vdots & \ddots & \ddots &      0 & \vdots \\
             0 &  \dots &      0 &      1 & -\frac{t_{g,d+1}}{s_{g,d}}
    \end{array}\right].
\end{equation*}
\end{theorem}

We come to the explicit classification of Fano simplices. In \cite{HaeHaHaSpr22} the authors present an efficient procedure for the classification of Fano triangles with fixed Gorenstein index based on unit fraction partitions. This procedure is completely automated and the authors carry out the classification of Fano triangles up to Gorenstein index $200$. We generalize and speed up their procedure, which allows us to efficiently classify Fano simplices of any given dimension and Gorenstein index. This allows us to carry out the following classifications; the complete classification data, as well as the Julia code \cite{Ju17} to produce these results can be found at \cite{Bae23}.

\begin{theorem}\label{thm:classf-d=2}
Up to isomorphy there are $2,\!992,\!229$ Fano triangles of Gorenstein index $g \le 1000$. The number of triangles $N(g)$ for given Gorenstein index $g$ develops as follows:

\medskip

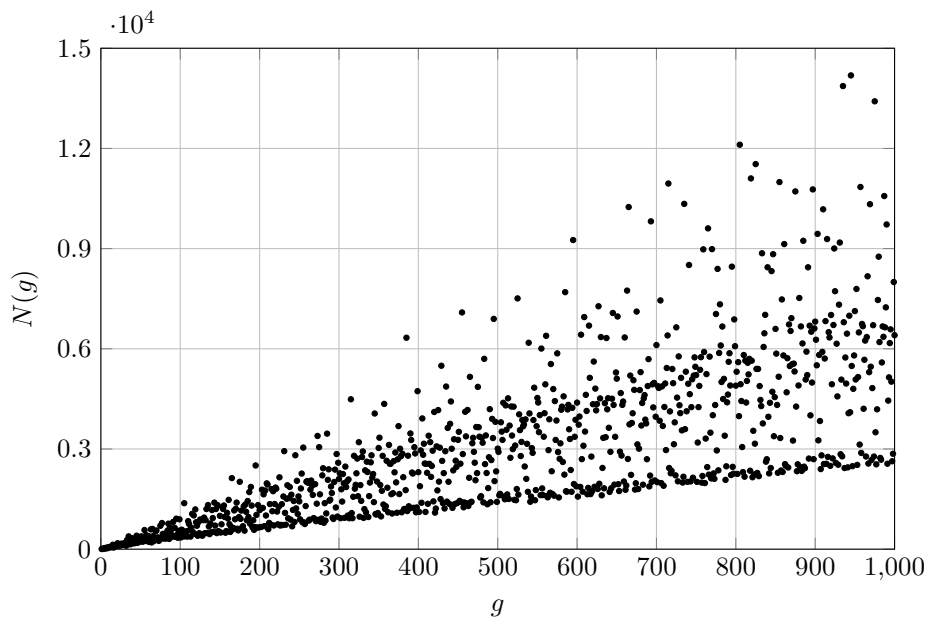
\begin{tikzpicture}
 
\begin{axis}[
    xmin = 0, xmax = 1000,
    ymin = 0, ymax = 15000,
    xtick distance = 100,
    ytick distance = 3000,
    grid = both,
    major grid style = {lightgray},
    width = 0.95\textwidth,
    height = 0.65\textwidth,
    xlabel = {$g$},
    ylabel = {$N(g)$},]
 
\addplot[
    only marks,
    mark = *,
    mark size = 1pt
] file[skip first] {fwps-d=2-data.txt};
 
\end{axis}
 
\end{tikzpicture}
\end{theorem}

\goodbreak

\begin{theorem}\label{thm:classf-d=3}
Up to isomorphy there are $9,\!368,\!501$ Fano simplices of dimension three and Gorenstein index $g \le 30$. The number of simplices $N(g)$ for given Gorenstein index $g$ develops as follows:
\setlength{\tabcolsep}{3pt}\begin{longtable}{c|cccccccc}
$g$ & $1$ & $2$ & $3$ & $4$ & $5$ & $6$ & $7$ & $8$ \\[3pt]
$N(g)$ & $48$ & $435$ & $1,\!703$ & $3,\!042$ & $7,\!506$ & $14,\!527$ & $16,\!627$ & $21,\!789$ \\[3pt]
\hline\hline\\[-10pt]
$g$ & $9$ & $10$ & $11$ & $12$ & $13$ & $14$ & $15$ & $16$ \\[3pt]
$N(g)$ & $39,\!288$ & $61,\!295$ & $54,\!404$ & $100,\!670$ & $59,\!500$ & $157,\!071$ & $269,\!037$ & $121,\!530$ \\[3pt]
\hline\hline\\[-10pt]
$g $ & $17$ & $18$ & $19$ & $20$ & $21$ & $22$ & $23$ & $24$ \\[3pt]
$N(g)$ & $133,\!559$ & $319,\!176$ & $173,\!707$ & $473,\!732$ & $523,\!939$ & $401,\!328$ & $332,\!612$ & $695,\!989$ \\[3pt]
\hline\hline\\[-10pt]
$g $ & $25$ & $26$ & $27$ & $28$ & $29$ & $30$ &  &  \\[3pt]
$N(g)$ & $515,\!042$ & $565,\!225$ & $824,\!950$ & $1,\!007,\!089$ & $513,\!356$ & $1,\!960,\!325$ &  & 
\end{longtable}
\end{theorem}

\begin{theorem}\label{thm:classf-d=4}
Up to isomorphy there are $87,532$ Fano simplices of dimension four and Gorenstein index $g \le 2$. Of those, $1561$ are of Gorenstein index $g = 1$. The remaining $85,971$ simplices are of Gorenstein index $g = 2$.
\end{theorem}

\begin{remark}
By the correspondence between Fano simplices and fake weighted projective spaces, Theorems \ref{thm:classf-d=2} - \ref{thm:classf-d=4} are also classifications of fake weighted projective spaces of corresponding dimension and Gorenstein index.
\end{remark}

Let us compare our results to existing classifications. In dimension two, Theorem~\ref{thm:classf-d=2} encompasses in particular the classification in \cite{Da09} of fake weighted projective planes of Gorenstein index at most three and the toric part of the classification in~\cite{HaeHaHaSpr22}. In dimension three we mention~\cite{Ka10}, where Kasprzyk classifies the three-dimensional \emph{canonical} Fano polytopes, ie. those with a single interior lattice point. The overlap with Theorem~\ref{thm:classf-d=3} consists of precisely~$204$ canonical Fano simplices of Gorenstein index at most $30$. There are only $21$ three-dimensional canonical Fano simplices that have Gorenstein index larger than~$30$. The largest Gorenstein index among those is $g = 420$. In dimension four we have the classification of the~$1561$ reflexive simplices by Kreuzer and Skarke~\cite{KrSk00}, which correspond to the~$1561$ Fano simplices of Gorenstein index one from Theorem \ref{thm:classf-d=4}. Note that there is no overlap between Theorem \ref{thm:classf-d=4} and the classification of empty $4$-simplices \cite{IgSa21} as our simplices have at least one interior lattice point.
Let us also mention Balletti's recent extensive classification of lattice polytopes by their volume \cite{Ba21}, where the polytopes are classified up to \emph{affine unimodular equivalence}, ie. also allowing for translations. As this does not leave the Gorenstein index invariant, their results are not immediately comparable to Theorems \ref{thm:classf-d=2} - \ref{thm:classf-d=4}.

Whereas the bounds for the volume and Mahler volume in Theorems \ref{thm:max-vol} and \ref{thm:max-mahler} are all sharp, we obtain sharpness in Theorem~\ref{thm:fwps-max-mult} (iii) for odd Gorenstein indices only. More explicitly we have the following.

\begin{remark}
In the case of even Gorenstein index $g$, the values provided for~$a_1,\dots,a_{d-1}$ in Theorem \ref{thm:fwps-max-mult} (iii) results in a matrix $P$ with the last column being non-primitive. In fact our classification results suggest, that for even Gorenstein index the multiplicity bound in Theorem \ref{thm:fwps-max-mult} (iii) is too high. We conjecture that in this case, apart from $(d,g) = (3,2), (3,4)$, the multiplicity is bounded by
\begin{equation*}
    \mult(\Delta) \ \le \ \frac{2t_{g,d-1}^2}{g}
\end{equation*}
and this bound is sharp, ie. there is a Fano simplex of dimension $d$ and Gorenstein index $g$ that attains this bound.
\end{remark}

The article is organized as follows. In Section \ref{sect:ratsimp} we associate with every IP simplex a unit fraction partition of its Gorenstein index. The main result of this section is Proposition \ref{prop:uf-vol}, which relates the volume and the multiplicity of a (Fano) simplex to its unit fraction partition. Section \ref{section:sharp-bounds-ufp} is dedicated to providing sharp bounds on unit fraction partitions. The main result is Theorem \ref{thm:ufp-bounds}, which is the foundation for proofing Theorems \ref{thm:max-vol} - \ref{thm:max-mahler}. Section \ref{sect:proofs} contains the proof of Theorems \ref{thm:max-vol} - \ref{thm:max-mahler}. In Section \ref{sec:fwps-classification} we present our classification procedure for Fano simplices.

\goodbreak

\section{Simplices, weight systems and unit fractions}\label{sect:ratsimp}

We associate with every IP simplex a unit fraction partition of its Gorenstein index, see Proposition \ref{prop:a-of-delta}.The main result of this section is Proposition \ref{prop:uf-vol}, which relates the volume and the factor of an IP simplex to its unit fraction partition. Throughout, $N$ is a rank $d$ lattice for some~$d \in \ZZ_{\ge 2}$. The dual lattice is denoted by $M = \Hom(N,\ZZ)$ with pairing $\bangle{\cdot\, , \cdot } \colon M \times N \rightarrow \ZZ$. We write $N_\QQ$ and $M_\QQ$ for the corresponding rational vector spaces. We assume that polytopes are always full dimensional, ie. $\lin(\Delta) = N_\QQ$. The \emph{normalized volume} of a $d$-dimensional polytope~$\Delta$ is~$\Vol(\Delta) = d!\, \vol(\Delta)$, where~$\vol(\Delta)$ denotes its euclidean volume. By an \emph{IP polytope} we mean a polytope $\Delta \subseteq N_\QQ$ that contains the origin $\mathbf{0} \in N_\QQ$ in its interior. The \emph{dual} of an IP polytope $\Delta$ is the polytope
\begin{equation*}
    \Delta^* \ := \ \{ u \in M_\QQ ; \, \bangle{u,v} \ge -1 \text{ for all } v \in \Delta\} \ \subseteq \ M_\QQ.
\end{equation*}
For a facet $F$ of $\Delta$ we denote by $u_F \in M_{\QQ}$ the unique linear form with $\bangle{u_F,v} = -1$ for all $v \in F$. We have
\begin{equation*}
    \Delta^* \ = \ \conv(\, u_F; \, F \text{ facet of } \Delta\, ), \qquad \Delta \ = \ \{v \in N_\QQ ; \, \bangle{u_F,v} \ge -1,\, F \text{ facet of } \Delta\}.
\end{equation*}
A \emph{lattice polytope} is a polytope $\Delta \in N_\QQ$ whose vertices lie in the lattice $N$. A \emph{Fano} polytope is an IP lattice polytope whose vertices are primitive lattice points. We regard two IP polytopes $\Delta \subseteq N_{\QQ}$ and $\Delta' \subseteq N'_{\QQ}$ as isomorphic if there is a lattice isomorphism~$\varphi \colon N \rightarrow N'$ mapping $\Delta$ bijectively to $\Delta'$.

\begin{definition}See~\cites{Co02,Ni07}.
A \emph{weight system $Q$} of length $d$ is a $(d+1)$-tuple of positive rational numbers $Q = (q_0, \dots, q_d)$. The \emph{total weight} of a weight system $Q$ is the rational number $|Q| := q_0 + \dots + q_d$. A weight system $Q$ is called \emph{reduced} if it consists of integers and $\gcd(q_0,\dots,q_d) = 1$ holds. A reduced weight system is called \emph{well-formed} if~$\gcd(q_j ; \, j = 0,\dots,d,\ j \ne i ) = 1$ holds for all $i = 0,\dots, d$. Any weight system $Q$ can be written as $\lambda(Q) \cdot Q^\red$ with a unique reduced weight system $Q^\red$ and a unique positive rational number $\lambda(Q)$. We call $\lambda(Q)$ the \emph{factor} of $Q$ and $Q^\red$ the \emph{reduction} of $Q$.
\end{definition}

\begin{definition}See~\cites{Co02,Ni07}.
To any IP simplex $\Delta$ with vertices~$v_0,\dots,v_d \in \QQ^d$ we associate a weight system by
\begin{equation*}
    Q_\Delta \ := \ (q_0,\dots,q_d), \qquad q_i \ := \ |\det(\, v_j ; \ j = 0,\dots, d, \ j \ne i\,)|.
\end{equation*}
\end{definition}

\begin{remark}\label{rem:ws-props}
Let $\Delta \subseteq N_\QQ$ a $d$-dimensional IP simplex with vertices $v_0,\dots,v_d$ and weight system $Q_\Delta = (q_0,\dots,q_d)$.
\begin{enumerate}
    \item For the total weight we have $|Q_\Delta| = \Vol(\Delta)$.
    \item If~$\Delta$ is a Fano simplex, then $Q_\Delta^\red$ is well-formed.
    \item We have $\sum_{i=0}^d q_i v_i = 0$ and $Q_\Delta^\red = (q'_0.\dots,q'_d)$ is the unique reduced weight system satisfying $\sum_{i=0}^d q'_i v_i = 0$.
    \item For any $H \in \GL(d,N_\QQ)$ we have $Q_{H\Delta} = |\det(H)|\, Q_\Delta$. In particular, the weight systems of isomorphic IP simplices coincide up to order.
\end{enumerate}
\end{remark}

For an IP lattice simplex $\Delta \subseteq N_\QQ$ we denote by $N(\Delta) \subseteq N$ the sublattice generated by the vertices of $\Delta$. If $\Delta \subseteq N_\QQ$ is any IP simplex and $\Delta' := g_\QQ(\Delta)\, \Delta$, then we have
\begin{equation*}
    \lambda(\Delta) \ := \ \lambda(Q_\Delta) \ = \ \frac{[N:N(\Delta')]}{g_\QQ(\Delta)^d}.
\end{equation*}
In case $\Delta$ is a Fano simplex, we write $\mult(\Delta) := \lambda(\Delta)$ and call it the \emph{multiplicity} of~$\Delta$. It coincides with the cardinality of the torsion part of the class group~$\Cl(Z)$ of the associated fake weighted projective space~$Z = Z(\Delta)$. The following Proposition is a reformulation of~\cite{BoBo92}*{Prop.~2}. Compare also \cite{Co02}*{4.4--4.6}.

\begin{proposition}\label{prop:Qred-simp}
To any reduced weight system $Q$ of length $d$ there exists a~$d$-dimensional IP lattice simplex $\Delta(Q) \subseteq \QQ^d$, unique up to an isomorphism, with~$Q_{\Delta(Q)} = Q$. For any IP simplex $\Delta \in \QQ^d$ with $(Q_\Delta)_\red = Q$ there is a linear map $H \in \GL(d,\QQ)$ whose determinant satisfies~$|\det(H)| = \lambda(\Delta)$, such that~$\Delta = H \,\Delta(Q)$ holds.
\end{proposition}

\begin{example}
Consider a reduced weight system $Q = (q_0,\dots,q_d)$ and assume that~$q_d = 1$ holds. This situation will appear frequently in later sections. We can immediately write down $\Delta(Q)$: It's vertices are given by the columns of the matrix
\begin{equation*}
    P
    \ = \ 
    \left[\begin{array}{ccccc}
        1      & 0      & \dots  & 0      & -q_0   \\
        0      & \ddots & \ddots & \vdots & \vdots \\
        \vdots & \ddots & \ddots & 0      & \vdots \\
        0      & \dots  & 0      & 1      & -q_{d-1}
    \end{array}\right].
\end{equation*}
\end{example}

\begin{definition}\label{def:gorind}
Let $\Delta \subseteq N_\QQ$ an IP simplex.
\begin{enumerate}
    \item
    The \emph{index of rationality of $\Delta$} is the positive integer
    \begin{equation*}
        g_\QQ(\Delta) \ := \ \min \{\, k \in \ZZ_{\ge 1} ; \ k \Delta \text{ is a lattice simplex}\,\}.
    \end{equation*}
    \item
    The \emph{Gorenstein index of $\Delta$} is the positive integer
    \begin{equation*}
        g(\Delta) \ := \ g_\QQ(\Delta) \cdot g_\QQ(\Delta^*).
    \end{equation*}
    \item
    Assume $\Delta$ is a lattice simplex. Denote by $u_0,\dots,u_d \in M_\QQ$ the vertices of the dual $\Delta^* \subseteq M_\QQ$. We call $u_k$ the \emph{$k$-th Gorenstein form} of $\Delta$. We define the~\emph{$k$-th local Gorenstein index} $g_k$ of $\Delta$ as the smallest positive integer such that~$g_k u_k \in M$ holds.
\end{enumerate}
\end{definition}

\begin{remark}
If $\Delta \subseteq N_\QQ$ is an IP lattice simplex with local Gorenstein indices~$g_0,\dots,g_d$, then we have $g(\Delta) = \lcm( g_0,\dots,g_d )$.
\end{remark}

\begin{definition}
Let $g \in \ZZ_{\ge 1}$. A tuple $A = (\alpha_1,\dots,\alpha_n) \in \ZZ^n_{\ge 1}$ is called a \emph{unit fraction partition} (ufp for short) \emph{of $g$} of length $n$ if the following holds:
\begin{equation*}
    \frac{1}{g} \ = \ \sum\limits_{k=1}^n \frac{1}{a_k}.
\end{equation*}
A tuple $A = (\alpha_1,\dots,\alpha_n) \in \ZZ^n_{\ge 1}$ is called a \emph{unit fraction partition} if it is a ufp of $g$ for some $g \in \ZZ_{\ge 1}$. For a unit fraction partition $A = (\alpha_1,\dots,\alpha_n)$ of $g$ we call
\begin{equation*}
    t_A \ := \ \lcm(\alpha_1,\dots,\alpha_n), \qquad \lambda(A) \ := \ \gcd(g,\alpha_1,\dots,\alpha_n), \qquad A^{\red} \ := \ A/\lambda(A)
\end{equation*}
the \emph{total weight}, the \emph{factor} and the \emph{reduction} of $A$, respectively. A unit fraction partition~$A$ is called \emph{reduced} if it coincides with its reduction. It is called \emph{well-formed} if $\alpha_i \mid \lcm(\alpha_j \, ; \, j \ne i)$ holds for all $i = 1,\dots,n$.
\end{definition}

\begin{proposition}\label{prop:a-of-delta}
See \cite{Bae22}*{Prop.~3.2}. Let $\Delta \subseteq N_\QQ$ a $d$-dimensional IP simplex of Gorenstein index $g$ with weight system $Q_\Delta = (q_0, \dots,q_d)$. Then
\begin{equation*}
        A(\Delta) \ := \ \left( \frac{g |Q_\Delta|}{q_0}, \dots, \frac{g |Q_\Delta|}{q_d} \right)
\end{equation*}
is a unit fraction partition of $g$ of length $d+1$. We call it the \emph{unit fraction partition of $g$ associated with $\Delta$}.
\end{proposition}

The following Proposition establishes a connection between geometric properties of an IP simplex $\Delta$ and its associated unit fraction partition. It can be seen as an extension of \cite{Ni07}*{Prop.~4.5} to the case of non-reflexive IP simplices. Compare also~\cite{Bae22}*{Prop.~3.3}.

\begin{proposition}\label{prop:uf-vol}
Let $\Delta \subseteq N_\QQ$ a $d$-dimensional IP simplex of Gorenstein index $g(\Delta) = g$ with associated unit fraction partition $A(\Delta) = (\alpha_0,\dots,\alpha_d)$ of $g$. Then~$A(\Delta) = A(\Delta^*)$ holds and we have:
\begin{enumerate}
    \item
    \begin{equation*}
        \Vol(\Delta)\Vol(\Delta^*)
        \ = \
        \frac{1}{g^{d+1}}\,\alpha_0\cdots \alpha_d,
    \end{equation*}
    
    \item
    \begin{equation*}
        \lambda(\Delta^*) \Vol(\Delta)
        \ = \
        \lambda(\Delta) \Vol(\Delta^*)
        \ = \
        \frac{1}{g^d} \frac{\alpha_0\cdots \alpha_d}{\lcm(\alpha_0,\dots,\alpha_d)},
    \end{equation*}
    
    \item
    \begin{equation*}
        \lambda(\Delta) \lambda(\Delta^*)
        \ = \
        \frac{1}{g^{d-1}} \frac{\alpha_0\cdots \alpha_d}{\lcm(\alpha_0,\dots,\alpha_d)^2}.
    \end{equation*}        
\end{enumerate}
\end{proposition}

Note that the left hand side of equations (i)-(iii) in Proposition \ref{prop:uf-vol} only depends on the simplex $\Delta$, while the right hand side only depends on the unit fraction partition $A(\Delta)$. For the proof of Proposition \ref{prop:uf-vol}, we need the following Lemma~\ref{lemma:ws-connection}, which is originally \cite{Ni07}*{Prop.~3.6}.

\begin{definition}
See \cite{Ni07}*{Def.~3.4}. For any weight system $Q = (q_0,\dots,q_d)$ set
\begin{equation*}
    m_Q \ := \frac{|Q|^{d-1}}{q_0 \cdots q_d}.
\end{equation*}
\end{definition}

\begin{lemma}\label{lemma:ws-connection}
See \cite{Ni07}*{Prop.~3.6}. For any $d$-dimensional IP simplex $\Delta$ we have
\begin{equation*}
Q_{\Delta^*} \ = \ m_{Q_\Delta} Q_{\Delta}.
\end{equation*}
\end{lemma}

\begin{proof}[Proof of Proposition \ref{prop:uf-vol}]
By Lemma \ref{lemma:ws-connection} the weight systems $Q_{\Delta}$ and $Q_{\Delta^*}$ differ only by a factor. Moreover, the simplices~$\Delta$ and $\Delta^*$ have the same Gorenstein index. Thus the associated unit fraction partitions $A(\Delta)$ and $A(\Delta^*)$ coincide. Item~(i) is an immediate consequence of (ii) and (iii). We prove (ii). Remark \ref{rem:ws-props} (i) together with Lemma \ref{lemma:ws-connection} yields
\begin{equation*}
    \Vol(\Delta^*)
    \ = \
    |Q_{\Delta^*}|
    \ = \
    \frac{|Q_\Delta|^d}{q_0 \cdots q_d}.
\end{equation*}
We multiply this by the multiplicity $\lambda(\Delta)$ and use the identity $\lambda(\Delta) = g|Q_\Delta|/t_{A(\Delta)}$ to obtain
\begin{equation*}
    \lambda(\Delta) \Vol(\Delta^*)
    \ = \
    \frac{g |Q_\Delta|}{t_{A(\Delta)}}\,\frac{|Q_\Delta|^d}{q_0 \cdots q_d}
    \ = \
    \frac{1}{g^d} \frac{\alpha_0 \cdots \alpha_d}{\lcm(\alpha_0,\dots,\alpha_d)}.
\end{equation*}
Switching the roles of $\Delta$ and $\Delta^*$ and using the fact that they have the same unit fraction partition, we obtain $\lambda(\Delta) \Vol(\Delta^*) = \lambda(\Delta^*) \Vol(\Delta)$. We prove (iii). Let~$Q_\Delta^\red = (q'_0, \dots, q'_d)$. With Lemma \ref{lemma:ws-connection} we obtain:
\begin{equation*}
    \lambda(\Delta^*) \ = \ \lambda(\Delta) m_{Q_\Delta} \ = \ \lambda(\Delta) \frac{|Q_\Delta|^{d-1}}{q_0 \cdots q_d} \ = \ \frac{1}{\lambda(\Delta)} \frac{|Q_{\Delta}^\red|^{d-1}}{q'_0 \cdots q'_d} \ = \ \frac{1}{\lambda(\Delta)} m_{Q_{\Delta}^\red},
\end{equation*}
Multiplying both sides by $\lambda(\Delta)$ yields the identity $\lambda(\Delta)\lambda(\Delta^*) = m_{Q_{\Delta}^\red}$. We obtain:
\begin{equation*}
    m_{Q_{\Delta}^\red}
    \ = \
    \frac{1}{|Q_{\Delta}^\red|^2} \frac{|Q_{\Delta}^\red|^{d+1}}{q'_0 \cdots q'_d}
    \ = \
    \frac{\lambda(\Delta)^2}{|Q_{\Delta}|^2} \frac{|Q_{\Delta}|^{d+1}}{q_0 \cdots q_d}
    \ = \
    \frac{1}{g^{d-1}} \frac{\alpha_0 \cdots \alpha_d}{\lcm(\alpha_0,\dots,\alpha_d)^2}.
\end{equation*}
\end{proof}

It will be convenient to assign unit fraction partitions directly to weight systems and vice versa.

\begin{definition}
The \emph{index of a weight system} $Q = (q_0,\dots,q_d)$ is the positive integer
\begin{equation*}
    g(Q) \ := \ \min \left( \, k \in \ZZ_{\ge 1} ; \ k|Q|/q_i \in \ZZ \text{ for all } i = 0,\dots,d \, \right).
\end{equation*}
\end{definition}

\begin{remark}
The index $g(Q_\Delta)$ of the weight system $Q_\Delta$ of an IP simplex $\Delta$ is always a divisor of it's Gorenstein index $g(\Delta)$. They might coincide, although frequently $g(Q_\Delta)$ is a true divisor of $g(\Delta)$.
\end{remark}

\begin{proposition}\label{prop:ws-ufp}
See \cite{Bae22}*{Prop.~3.6}. Let $Q = (q_0,\dots,q_d)$ a weight system of length $d$ and let $A = (\alpha_0,\dots,\alpha_d)$ a unit fraction partition of $g \in \ZZ_{\ge 1}$ of length~$d+1$. Set
\begin{equation*}
    A(Q)
    \ := \
    \left( \frac{g |Q|}{q_0}, \dots, \frac{g |Q|}{q_d} \right),
    \qquad\qquad
    Q(A)
    \ := \
    \left( \frac{t_A}{\alpha_0}, \dots, \frac{t_A}{\alpha_d} \right).
\end{equation*}
Then $A(Q)$ is a reduced unit fraction partition of $g(Q)$ and $Q(A)$ is a reduced weight system of length $d$ and index $g(Q(A)) = g/\lambda(A)$. Moreover, we have
\begin{equation*}
    Q(A(Q)) \ = \ Q^\red, \qquad\qquad A(Q(A)) \ = \ A^\red
\end{equation*}
and this correspondence respects well-formedness.
\end{proposition}

\goodbreak

\section{Sharp bounds on unit fraction partitions}\label{section:sharp-bounds-ufp}

For a unit fraction partition $A = (\alpha_1,\dots,\alpha_n)$ of length~$n$ we consider
\begin{equation*}
    F_k(A) \ := \ \frac{\alpha_1 \cdots \alpha_n}{\lcm(\alpha_1,\dots,\alpha_n)^{n-k}}.
\end{equation*}
This is the expression that shows up in Proposition \ref{prop:uf-vol} for~$k=n-2,n-1$ and~$n$. We give sharp bounds on these expressions among all unit fraction partitions of $g$ and completely describe the unit fraction partitions attaining those bounds.

\begin{definition}
The \emph{$g$-Sylvester sequence} $S_g = (s_{g,1},s_{g,2},\dots)$ and the \emph{truncated~$g$-Sylvester sequence} $T_g = (t_{g,1},t_{g,2},\dots)$ for a positive integer $g$ are given by
\begin{equation*}
    s_{g,1} \ := \ g+1, \qquad s_{g,k+1} \ := \ s_{g,k} (s_{g,k} - 1) + 1, \qquad t_{g,k} := s_{g,k} - 1.
\end{equation*}
\end{definition}

\begin{theorem}\label{thm:ufp-bounds}
Let $g \ge 1$ and let $n \ge 3$. Let~$A = (\alpha_1,\dots,\alpha_n)$ a unit fraction partition of~$g$ with $\alpha_1 \le \dots \le \alpha_n$. For the value of $F_k$ on $A$ the following hold:
\begin{enumerate}
    \item
    $F_n(A) \le t_{g,n}^2/g$ and equality holds if and only if $A = (s_{g,1},\dots,s_{g,n-1},t_{g,n})$.
    
    \item
    Assume $k = n-1$. If $(n,g) = (3,1)$, then we have $F_{n-1}(A) \le 9$ and equality holds if and only if~$A = (3,3,3)$. In all other cases we have
    \begin{equation*}
        F_{n-1}(A) \ \le \ \frac{2\, t_{g,n-1}^2}{g}.
    \end{equation*}
    Equality holds if and only if $A$ is one of the following unit fraction partitions:
    \begin{equation*}
        \qquad\qquad
        (6,6,6),
        \qquad
        (2,6,6,6),
        \qquad
        (s_{g,1}, \dots, s_{g,n-2}, 2t_{g,n-1},2t_{g,n-1}).
    \end{equation*}
     
    \item
    Assume $k = n-2$. If $n = 4$ and $g \in \{1,2\}$, then we have $F_{n-2}(A) \ \le \ 16g^2$ and equality holds if and only if~$A = (4g,4g,4g,4g)$. If $(n,g) = (5,1)$, then we have $F_{n-2}(A) \ \le \ 128$ and equality holds if and only if~$A = (2,8,8,8,8)$.
    In all other cases we have
    \begin{equation*}
        F_{n-2}(A) \ \le \ \frac{3\, t_{g,n-2}^2}{g}.
    \end{equation*}
    Equality holds if and only if $A$ is one of the following unit fraction partitions:
    \begin{equation*}
        \qquad\qquad
        (12,12,12,12),
        \qquad
        (s_{g,1}, \dots, s_{g,n-3}, 3t_{g,n-2}, 3t_{g,n-2}, 3t_{g,n-2}).
    \end{equation*}
\end{enumerate}
\end{theorem}

\begin{remark}\label{rem:syl-prop}
In the literature the sequence $S_1 = (s_{1,1},s_{1,2},\dots)$ is known as \emph{Sylvester's sequence}, see for instance \cite{OEIS-syl}. Our naming for the sequences~$S_g$ and $T_g$ is derived from that. We list some properties of the sequences~$S_g$ and $T_g$ that we will use frequently.
\begin{enumerate}
    \item
    For any $n \ge 1$ we have
    \begin{equation*}
        \frac{1}{g} \ = \ \frac{1}{s_{g,1}} + \dots + \frac{1}{s_{g,n-1}} + \frac{1}{t_{g,n}}.
    \end{equation*}
    
    \item
    For any $n \ge 1$ we have
    \begin{equation*}
        \frac{g}{t_{g,n}} \ = \ \frac{1}{s_{g,1}} \cdots \frac{1}{s_{g,n-1}}.
    \end{equation*}

    \item
    For any $g,n \ge 1$ we have $s_{g,n+1} > s_{g,n}$ and $s_{g+1,n} > s_{g,n}$.
    
    \item
    For $i\ne j$ we have $\gcd(s_{g,i},s_{g,j}) = 1$.
\end{enumerate}
\end{remark}

The strategy for the proof of Theorem \ref{thm:ufp-bounds} is as follows: For given $g$ and $n$ we define a certain compact subset~$A_g^n \subseteq \RR^n$, which has the property that for any unit fraction partition $A = (\alpha_1,\dots,\alpha_n)$ of~$g$ with~$\alpha_1 \le \dots \le \alpha_n$, the point~$(1/\alpha_1,\dots,1/\alpha_n)$ is contained in $A_g^n$. For $k \in \{n-2,n-1,n\}$ we minimize the function $f_k(x) := x_1\cdots x_k$ on $A_g^n$ and show that it attains its minimum precisely at the points corresponding to the unit fraction partitions listed in Theorem \ref{thm:ufp-bounds}. This strategy for minimizing functions on unit fraction partitions was first used by Izhboldin and Kurliandchik in \cite{IzKu95}, see also~\cites{AvKrNi15,Ni07} for generalizations. In \cite{AvKrNi15} the authors call this type of optimization problems \emph{Izhboldin-Kurliandchik problems}. In the following we adopt their naming convention.

\begin{definition}\label{def:sets-Agn}
Let $g,n \ge 1$. We denote by $A_g^n \subseteq \RR^n$ the compact set of all points~$(x_1,\dots,x_n) \in \RR^n$ that satisfy the following conditions:
\begin{enumerate}
    \item[(A1)] $x_1 \ge \dots \ge x_n \ge 0$.
    \item[(A2)] $x_1 + \dots + x_n = 1/g$.
    \item[(A3)] $x_1\cdots x_k \le g (x_{k+1} + \dots + x_n)$ for all $k = 1, \dots, n-1$.
\end{enumerate}
For $x \in \RR^n$ we denote by $\SUM(g)$ the equality $x_1+\dots+x_n = 1/g$, by~$\ORD(k)$ we denote the inequality $x_k \ge x_{k+1}$ and by $\PS(g,k)$ the inequality~$x_1\cdots x_k~\le~g (x_{k+1}~+~\dots~+~x_n)$. Thus the set $A_g^n$ consists of the points~$(x_1,\dots,x_n) \in \RR^n_{\ge 0}$ that satisfy the equality $\SUM(g)$ and the inequalities~$\ORD(k)$ and $\PS(g,k)$ for all $k=1,\dots,n-1$.
\end{definition}

\begin{lemma}\label{lemma:ufp_in_Agn}
See \cite{Bae22}*{Lemma~4.3}. Let $g \ge 1$ and $n \ge 1$. For any unit fraction partition~$A = (\alpha_1,\dots,\alpha_n)$ of $g$ with~$\alpha_1 \le \dots \le \alpha_n$ the point $(1/\alpha_1,\dots,1/\alpha_n)$ is contained in $A_{g}^n$.
\end{lemma}

In the following Proposition we gather important properties of the set $A_g^n$. It is a generalization of~\cite{AvKrNi15}*{Lemma~4.1}.

\begin{proposition}\label{prop:structure-Agn}
Let $g \ge 1$ and $n \ge 3$. For any point $x \in A_g^n$ the following hold:
\begin{enumerate}
    \item We have $1/g > x_1 \ge x_n > 0$.
    \item The inequalities $\ORD(k)$ and $\PS(g,k)$ cannot simultaneously be fulfilled with equality.
    \item If for some $1 \le k \le n-1$ the inequality $\PS(g,i)$ is fulfilled with equality for all~$i \le k$, then~$x_i = 1/s_{g,i}$ holds for all $i = 1,\dots,k$.
\end{enumerate}
\end{proposition}

\begin{proof}
We prove (i). Since the $x_i$ are all non-negative, the equality $\SUM(g)$ implies that~$x_1 \le 1/g$ holds. Assume $x_n = 0$. Then by the inequality $\PS(g,n-1)$ we have~$x_1 \cdots x_{n-1} = 0$. Thus $x_i = 0$ holds for some $i=1,\dots,n-1$. The inequality~$\ORD(j)$ thus implies that $x_j = 0$ holds for all $j=i,\dots,n$. We can repeat this argument to obtain~$x_1 = \dots = x_n = 0$, which contradicts the equality~$\SUM(g)$. Thus~$x_n > 0$ holds. We prove (ii). Assume that for some $k$ the inequalities $\ORD(k)$ and $\PS(g,k)$ hold simultaneously with equality. Using $x_{k+1} = x_k$, we may then write
\begin{equation*}
    0 \ = \ g (x_{k+1} + \dots + x_n) - x_1 \cdots x_k \ = \ g( x_{k+2} + \dots + x_n ) + x_k (g - x_1 \cdots x_{k-1}).
\end{equation*}
The first summand on the right hand side is non-negative, the second summand is positive. This means that they cannot add to zero, a contradiction. Thus $\ORD(k)$ and $\PS(g,k)$ cannot simultaneously be fulfilled with equality. We prove (iii) by induction on $i$. In case~$i = 1$, if $\PS(g,1)$ holds with equality, we get
\begin{equation*}
    x_1 \ = \ g ( x_2 + \dots + x_n ) \ = \ g \left( \frac{1}{g} - x_1 \right).
\end{equation*}
Solving this equation for $x_1$ yields $x_1 = 1/s_{g,1}$. Now assume $i > 1$ and $x_j = 1/s_{g,j}$ holds for all $j \le i$. If $\PS(g,i)$ holds with equality, we obtain
\begin{equation*}
    \frac{g}{t_{g,i}} x_i \ = \ x_1 \cdots x_i \ = \ g ( x_{i+1} + \dots + x_{n} )
    \ = \ g \left( \frac{1}{g} - x_1 - \dots - x_i \right) \ = \ g \left( \frac{1}{t_{g,i}} - x_i \right).
\end{equation*}
Solving this equation for $x_i$ yields $x_i = 1/s_{g,i}$. This completes the proof.
\end{proof}

\begin{definition}
Let $g \ge 1$ and $n \ge 3$. For $k = 1, \dots, n$ we define a function $f_k$ by
\begin{equation*}
    f_k \colon A_g^n \ \rightarrow \ \RR, \qquad x \ \mapsto \ x_1 \cdots x_k.
\end{equation*}
Moreover, for $g,n$ and $k$ as above, we set
\begin{equation*}
    y(g,n,k) \ := \ \left( \ \frac{1}{s_{g,1}} \ , \ \dots \ , \ \frac{1}{s_{g,k-1}} \ , \ \frac{1}{(n-k+1) t_{g,k}} \ , \ \dots \ , \ \frac{1}{(n-k+1) t_{g,k}} \ \right).
\end{equation*}
Note that $y(g,n,k)$ belongs to $A_g^n$.
\end{definition}

\begin{proposition}\label{prop:fk-minimum-general}
Let $g \ge 1$, $n \ge 3$ and $k \in \{1,\dots,n\}$. Let $y = (y_1,\dots,y_n) \in A_g^n$ such that $f_k$ attains its minimum at $y$. Then $y = y(g,n,i_0)$ holds for some $i_0 \le k$.
\end{proposition}

The major part of the proof of Proposition \ref{prop:fk-minimum-general} is governed by the following Lemma.

\begin{lemma}\label{lemma:structure-minimum-fk}
Let $n \ge 3$ and $k \in \{1,\dots,n\}$. Let $y = (y_1,\dots,y_n) \in A_g^n$ such that~$f_k$ attains its minimum at $y$. Denote by $i_0$ the minimal index such that $y_{i_0} = y_n$ holds. Then the following hold:
\begin{enumerate}
    \item $i_0 \le k$.
    \item The inequality $\ORD(i)$ is strict for all $1 \le i \le i_0-1$.
    \item The inequality $\PS(g,i)$ holds with equality for all $1 \le i \le i_0-1$.
\end{enumerate}
\end{lemma}

\begin{proof}
The strategy for proving (i)-(iii) is the same in each of the three cases. We will assume that the assertion is false and this will allow us to construct a point~$y' \in A_g^n$ with~$f_k(y') < f_k(y)$, which contradicts the choice of $y$. Thus the assertion must be true.
We prove (i). For $k = n$ there is nothing to prove. Let~$k < n$. Assume that $i_0 > k$ holds. Let $j_0$ maximal with $y_k = y_{j_0}$. By assumption $j_0 < i_0$ holds. The entries of $y$ satisfy
\begin{equation*}
    y_1 \ \ge \ \dots \ \ge \ y_k \ = \ \dots \ = \ y_{j_0} \ > \ y_{j_0 + 1} \ \ge \ \dots \ \ge \ y_{i_0 - 1} \ > \ y_{i_0} \ = \ \dots \ = \ y_n.
\end{equation*}
We first consider the case $i_0 = n$. In that case we have $y_{n-1} > y_n$. Let
\begin{equation*}
    0 \ < \ \epsilon \ < \ \min \left( \frac{y_{j_0+1} - y_{j_0}}{2}, \frac{y_{n-1} - y_{n}}{2(j_0-k+1)} \right),
\end{equation*}
\begin{equation*}
    y' \ = \ (y_1,\dots,y_{k-1},y_k - \epsilon, \dots, y_{j_0} - \epsilon, y_{j_0+1}, \dots, y_{n-1}, y_n + \tilde{\epsilon}),
\end{equation*}
where $\tilde{\epsilon} = (j_0-k+1)\, \epsilon$. We show that $y'$ lies in $A_g^n$. By the choice of $\epsilon$ and $\tilde{\epsilon}$, the equality~$\SUM(g)$ holds for $y'$ and the inequality $\ORD(i)$ holds for $y'$ for all $i$. Moreover, for all $i \le n-1$ we have
\begin{equation*}
    y'_1 \cdots y'_i \ \le \ y_1 \cdots y_i \ \le \ g(y_{i+1} + \dots + y_n) \ \le \ g(y'_{i+1} + \dots + y'_n),
\end{equation*}
thus $\PS(g,i)$ holds as well. This shows that $y'$ lies in $A_g^n$. Evaluating $f_k$ on $y'$ we obtain
\begin{equation*}
    f_k(y') \ = \ y'_1 \cdots y'_k \ = \ y_1 \cdots y_{k-1} \cdot (y_k - \epsilon)  \ < \ f_k(y),
\end{equation*}
which contradicts the choice of $y$. Now assume that $i_0 < n$ holds. Note that since~$\ORD(i)$ holds with equality for all $i \ge i_0$, by Proposition \ref{prop:structure-Agn} (ii) the inequality~$\PS(g,i)$ is strict for $y$. Thus for each $i \ge i_0$ we can find $\delta_i > 0$, such that
\begin{equation*}
    y_1 \cdots y_{i_0-1} \cdot (y_{i_0} + \delta_i) \cdots (y_{i} + \delta_i)
    \ < \ 
    g(y_{i+1} + \dots + y_n) - g(n-i)\delta_i.
\end{equation*}
We denote by $\delta$ the minimum of all the $\delta_i$. Let
\begin{equation*}
    0 \ < \ \epsilon \ < \ \min \left( \frac{y_{j_0+1} - y_{j_0}}{2}, \frac{y_{i_0-1} - y_{i_0}}{2}, \frac{n-i_0+1}{j_0-k+1}\delta \right),
\end{equation*}
\begin{equation*}
    y' \ = \ (y_1,\dots,y_{k-1},y_k - \epsilon, \dots, y_{j_0} - \epsilon, y_{j_0+1}, \dots, y_{i_0-1}, y_{i_0} + \tilde{\epsilon}, \dots, y_n + \tilde{\epsilon}),
\end{equation*}
where $\tilde{\epsilon} = \frac{j_0-k+1}{n-i_0+1}\, \epsilon$. Again, $\epsilon$ and $\tilde{\epsilon}$ are chosen such that $y'$ satisfies equality $\SUM(g)$ and the inequality $\ORD(i)$ for all $i$. We show that $\PS(g,i)$ holds for $y'$. This is clear for $i < i_0$. For $i \ge i_0$ note that $\tilde{\epsilon} < \delta \le \delta_i$ holds. Thus we have
\begin{equation*}
    y'_1 \cdots y'_i \ < \ y_1 \cdots y_{i_0-1} \cdot (y_{i_0} + \delta_i) \cdots (y_i + \delta_i) \ < \ g(y'_{i+1} + \dots + y'_n).
\end{equation*}
This shows that $y'$ belongs to $A_g^n$. Evaluating $f_k$ on $y'$ we again obtain the inequality $f_k(y') < f_k(y)$, which contradicts the choice of $y$. Thus $i_0 \le k$ holds, which proves~(i).

We prove (ii). Note that by definition of $i_0$ we have $y_{i_0-1} < y_{i_0}$, so for~$i_0 \le 2$ there is nothing to show. Assume $i_0 \ge 3$ holds. This also means we have $n \ge k \ge 3$. We show that $\ORD(i)$ is strict for all $i \le i_0-2$. Assume on the contrary that $y_i = y_{i+1}$ holds for some $i \le i_0-2$. Let $j_0 \le i$ maximal with $y_{j_0} = y_i$ and $j_1 \ge i$ minimal with $y_i = y_{j_1}$. We have $1 \le j_0 \le i < j_1 \le i_0-1$. For the entries of $y$ we have:
\begin{equation*}
    y_{j_0-1} \ > \ y_{j_0} \ = \ \dots \ = \ y_{i} \ = \ \dots \ = \ y_{j_1} \ > \ y_{j_1+1} \ \ge \ \dots \ \ge \ y_{i_0-1} \ > \ y_{i_0}.
\end{equation*}
By Proposition \ref{prop:structure-Agn} The inequality $\PS(g,l)$ is strict for all $j_0 \le l < j_1$. There is thus~$\delta > 0$ such that the following inequality holds for all $j_0 \le l < j_1$:
\begin{equation*}
    y_1 \cdots y_{j_0-1} \cdot (y_{j_0} + \delta) \cdot y_{j_0+1} \cdots y_{l} \ < \ g(y_{l+1} + \dots + y_n - \delta)
\end{equation*}
With this value $\delta$ we may choose an $\epsilon > 0$ as follows and define a point $y'$ depending on this $\epsilon$:
\begin{equation*}
    0 \ < \ \epsilon \ < \ \min\left( \frac{y_{j_0-1} - y_{j_0}}{2}, \frac{y_{j_1} - y_{j_1-1}}{2}, \delta \right),
\end{equation*}
\begin{equation*}
    y' \ = \ \left( y_1, \dots, y_{j_0} + \epsilon, \dots, y_{j_1} - \epsilon, \dots y_n \right).
\end{equation*}
We show that $y'$ lies in $A_g^n$. Clearly $y'$ satisfies $\SUM(g)$ and $\ORD(l)$ for all $l$. By the choice of $\epsilon$, the inequality $\PS(g,l)$ holds for $y'$ for $l < j_1$. For $l \ge j_1$ note that
\begin{equation*}
    y'_{j_0} \cdot y'_{j_1} \ = \ (y_{j_0} + \epsilon)(y_{j_1} - \epsilon) \ = \ y_{j_0}y_{j_1} - \epsilon^2 \ < \ y_{j_0}y_{j_1}.
\end{equation*}
Thus in this case $\PS(g,l)$ holds for $y'$ as well. This shows that $y'$ belongs to $A_g^n$. As before, we have $f_k(y') < f_k(y)$, which contradicts the choice of $y$. Thus $\ORD(i)$ is strict, which proves (ii).

We prove (iii). First we assume there is $i \le i_0 - 2$ such that $\PS(g,i)$ is strict. Then there is $\delta > 0$ such that the following inequality holds:
\begin{equation*}
    y_1 \cdots y_{i-1} \cdot (y_i + \delta) \ < \ g( y_{i+1} + \dots + y_n - \delta ).
\end{equation*}
With this value $\delta$ we may choose an $\epsilon > 0$ as follows and define a point $y'$ depending on this $\epsilon$:
\begin{equation*}
    0 \ < \ \epsilon \ < \ \min \left( \frac{y_{i-1} - y_i}{2}, \frac{y_i - y_{i+1}}{2}, \delta \right),
\end{equation*}
\begin{equation*}
    y' \ = \ (y_1, \dots, y_{i-1}, y_i + \epsilon, y_i - \epsilon, y_{i+1}, \dots , y_n).
\end{equation*}
With the same arguments as in (i) and (ii) we see that $y'$ lies in $A_g^n$ and again~$f_k(y') < f_k(y)$ holds, contradicting the choice of $y$. Thus $\PS(g,i)$ holds with equality. Now assume that~$\PS(g,i_0-1)$ is strict. For $t \in \RR$ consider the points
\begin{equation*}
    \tilde{y}(t) \ := \ (y_1,\dots,y_{i_0-2},y_{i_0-1} + t, y_{i_0} - \tilde{t}, \dots , y_{n} - \tilde{t}),
\end{equation*}
where $\tilde{t} = \frac{t}{n-i_0+1}$. Note that $\tilde{y}(0) = y$ holds. We define a function $f\colon \RR \rightarrow \RR$ by
\begin{equation*}
    f(t) \ := \ f_k(\tilde{y}(t)) \ = \ y_1 \cdots y_{i_0-2} (y_{i_0-1}+t) (y_{i_0}-\tilde{t})^{k-i_0+1}.
\end{equation*}
The derivative of $f$ is given by
\begin{equation*}
    f'(t) \ = \ y_1 \cdots y_{i_0-2} (y_{i_0-1}-\tilde{t})^{k-i_0} \left[ \left( y_{i_0} - \frac{k-i_0+1}{n-i_0+1} y_{i_0-1} \right) - (k-i_0+2) \tilde{t} \right].
\end{equation*}
Note that for $t$ close to zero, the factor before the square brackets is positive. The behaviour of $f$ close to $t=0$ is thus governed by the term
\begin{equation*}
    \delta \ = \ y_{i_0} - \frac{k-i_0+1}{n-i_0+1} y_{i_0-1}.
\end{equation*}
If $\delta$ is negative, then $f$ is monotone decreasing in a neighborhood of $t=0$. We can thus find $t > 0$ with $\tilde{y}(t) \in A_g^n$ such that $f(t) < f(0)$ holds. On the other hand, if~$\delta$ is positive, then $f$ is monotone decreasing in a neighborhood of $t=0$ and we can find~$t < 0$ with~$\tilde{y}(t) \in A_g^n$ and $f(t) < f(0)$. If $\delta = 0$, then $f$ has a local maximum at~$t=0$. There is thus a neighborhood of $t=0$ with $\tilde{y}(t) \in A_g^n$ and we have~$f(t) < f(0)$ for all~$t\ne0$ in that neighborhood. In all cases there is a point $\tilde{y}(t) \in A_g^n$ with~$f_k(\tilde{y}(t)) < f_k(y)$. A contradiction to the choice of $y$, thus~$\PS(g,i_0-1)$ holds with equality for $y$.
\end{proof}

\begin{proof}[Proof of Proposition \ref{prop:fk-minimum-general}]
Let $y = (y_1,\dots,y_n) \in A_g^n$ such that $f_k$ attains its minimum at $y$. Let $i_0$ minimal such that $y_{i_0} = y_n$ holds. By Lemma \ref{lemma:structure-minimum-fk} (i) we have~$i_0 \le k$. We show that $y = y(g,n,i_0)$ holds. By Lemma \ref{lemma:structure-minimum-fk} (ii) the inequality~$\PS(g,i)$ holds with equality for all $i=1\dots i_0 -1$. Proposition \ref{prop:structure-Agn} (iii) then tells us that $y_i = 1/s_{g,i}$ holds. Now using $y_{i_0} = \dots = y_n$ and the fact that $\PS(g,i_0-1)$ holds with equality, we obtain
\begin{equation*}
    \frac{g}{t_{g,i_0}} \ = \ \frac{1}{s_{g,1}} \cdots \frac{1}{s_{g,i_0-1}} \ = \ g(y_{i_0} + \dots + y_n) \ = \ g(n-i_0+1)y_{i_0}.
\end{equation*}
This shows that $y = y(g,n,i_0)$ holds, which completes the proof.
\end{proof}

\begin{proposition}\label{prop:fk-minimum}
Let $g \ge 1$, $n \ge 3$ and let $k \in \{1,\dots,n\}$. Let $y \in A_g^n$ such that~$f_k$ attains its minimum at $y$.
\begin{enumerate}
    \item
    Assume $k = n$. Then we have $f_n(y) = g/t_{g,n}^2$ and $y = y(g,n,n)$ holds.
    
    \item
    Assume $k = n-1$. If $(n,g) = (3,1)$, then we have $f_2(y) = 1/9$ and~$y$ is the point $y(1,3,1)$. In all other cases we have
    \begin{equation*}
        f_{n-1}(y) \ = \ \frac{g}{2t_{g,n-1}^2}
    \end{equation*}
    and the point $y$ is one of the following:
    \begin{equation*}
        y(2,3,1),
        \qquad
        y(1,4,2),
        \qquad
        y(g,n,n-1).
    \end{equation*}
    
    \item
    Assume $k = n-2$. If $n = 4$ and $g \in \{1,2\}$, then we have $f_{2}(y) = 1/(16g^2)$ and $y = y(g,4,1)$ holds. If $(n,g) = (5,1)$, then we have $f_{3}(y) = 1/128$ and~$y = y(1,5,2)$ holds.
    In all other cases we have
    \begin{equation*}
        f_{n-2}(y) \ = \ \frac{g}{3t_{g,n-2}^2}
    \end{equation*}
    and either $y = y(3,4,1)$ or $y = y(g,n,n-2)$.
\end{enumerate}
\end{proposition}

For the proof of Proposition \ref{prop:fk-minimum} we need the following Lemma.

\begin{lemma}\label{lemma:syl-ineq-technical}
Let $g \ge 1$ and $n \ge 3$. Then the following hold:
\begin{enumerate}
    \item
    For all $1 \le r \le n$ we have
    \begin{equation*}
        r^r t_{g,n-r+1}^{r+1} \ \le \ t_{g,n}^2.
    \end{equation*}
    Equality holds if and only if $r = 1$.
    \item
    Assume $(n,g) \ne (3,1)$. Then for all $1 \le r \le n-1$ we have
    \begin{equation*}
        (r+1)^r t_{g,n-r}^{r+1} \ \le \ 2 t_{g,n-1}^2.
    \end{equation*}
    Equality holds if and only if $r = 1$ or $(g,r,n)$ equals $(1,2,4)$ or $(2,2,3)$.
    \item
    Assume $(n,g) \not\in \{(4,1), (4,2), (5,1)\}$. Then for all $1 \le r \le n-2$ we have
    \begin{equation*}
        (r+2)^r t_{g,n-r-1}^{r+1} \ \le \ 3 t_{g,n-2}^2.
    \end{equation*}
    Equality holds if and only if $r = 1$ or $(g,r,n) = (3,2,4)$.
\end{enumerate}
\end{lemma}

\begin{proof}
We prove the assertions (i)-(iii) by induction on $r$ and $n$. Note that for $r = 1$ the inequalities in (i) - (iii) even hold with equality for any $n \ge 3$. We may thus assume~$r \ge 2$. Moreover, we will use the following, which can be verified by direct computation:
\begin{enumerate}
    \item[(a)] $(r+1)/r \le s_{g,n}$ holds for all values of $g,n$ and $r$.
    \item[(b)] If $n \ge 3$, then $(r+1)^2/r \le s_{g,n}$ holds for all $g$ and all $1 \le r \le n$.
    \item[(c)] If $n \ge 4$, or $n \ge 3$ and $g \ge 2$, then $(r+1)^2/r \le s_{g,n-1}$ holds for all $1 \le r \le n$.
    \item[(d)] If $n \ge 6$, or $n \ge 4$ and $g \ge 2$, then $(r+1)^2/r \le s_{g,n-2}$ holds for all $1 \le r \le n$.
\end{enumerate}
We prove (i). The cases $(r,n) = (2,3)$ and $(r,n) = (3,3)$ are verified by direct computation. In these two cases, the inequality is strict. Assume the assertion is true for a fixed pair~$(r,n)$. Then we have:
\begin{align*}
    (r+1)^{r+1} t_{g,(n+1)-(r+1)+1}^{(r+1)+1} \ &= \ r^r t_{g,n-r+1}^{r+1} (r+1)\left(\frac{r+1}{r}\right)^r t_{g,n-r+1} \\
    &\le \ t_{g,n}^2 (r+1)\left(\frac{r+1}{r}\right)^r t_{g,n-r+1} \\
    &\le \ t_{g,n}^2 s_{g,n} s_{g,n-1} \cdots s_{g,n-r+1} t_{g,n-r+1} \\
    &= \ t_{g,n}^2 t_{g,n+1} \\
    &< \ t_{g,n+1}^2.
\end{align*}
In the second step we used the induction hypothesis for the pair $(r,n)$ and in the third step we used (a) and (b). Thus the inequality (i) holds for the pair $(r+1,n+1)$ and it is strict in this case.

We prove (ii). The cases $(g,r,n) = (1,2,4)$ and $(g,r,n) = (1,3,4)$ as well as the cases~$(g,r,n) = (g,2,3)$ for all $g \ge 2$ are verified by direct computation. Here (ii) holds with equality for $(g,r,n) = (1,2,4)$ and for $(g,r,n) = (2,2,3)$ and is strict otherwise. Assume the assertion is true for a fixed pair $(r,n)$. Then we have:
\begin{align*}
    ((r+1)+1)^{r+1} t_{g,n-r}^{(r+1)+1} \ &= \ (r+1)^r t_{g,n-r}^{r+1} (r+2)\left(\frac{r+2}{r+1}\right)^r t_{g,(n+1)-(r+1)} \\
    &\le \ 2t_{g,n-1}^2 \frac{(r+2)^2}{r+1}\left(\frac{r+2}{r+1}\right)^{r-1} t_{g,n-r} \\
    &\le \ 2t_{g,n-1}^2 s_{g,n-1} s_{g,n-2} \cdots s_{g,n-r} t_{g,n-r} \\
    &= \ 2t_{g,n-1}^2 t_{g,n} \\
    &< \ 2t_{g,n}^2 \\
    &= \ 2t_{g,(n-1)+1}^2.
\end{align*}
In the second step we used the induction hypothesis for the pair $(r,n)$ and in the third step we used (a) and (c). Thus the inequality (ii) holds for the pair $(r+1,n+1)$ and it is strict in this case.

We prove (iii). For $n = 3$ there is nothing to prove. The cases $(g,r,n) = (g,2,4)$ for~$g \ge 3$, as well as $(g,r,n) = (2,r,5)$ and $(g,r,n) = (1,r,6)$ for $2 \le r \le n-2$ are verified by direct computation. Here (iii) holds with equality for $(g,r,n) = (3,2,4)$ and is strict otherwise. Assume the assertion is true for a fixed pair $(r,n)$. We may assume that $g \ge 2$ and $n \ge 4$, or $g = 1$ and $n \ge 6$. Then we have:
\begin{align*}
    ((r+1)+2)^{r+1} t_{g,(n+1)-(r+1)-1}^{(r+1)+1} \ &= \ (r+2)^r t_{g,n-r-1}^{r+1} (r+3)\left(\frac{r+3}{r+2}\right)^r t_{g,n-r-1} \\
    &\le \ 3t_{g,n-2}^2 \frac{(r+3)^2}{r+2}\left(\frac{r+3}{r+2}\right)^{r-1} t_{g,n-r-1} \\
    &\le \ 3t_{g,n-2}^2 s_{g,n-2} s_{g,n-3} \cdots s_{g,n-r-1} t_{g,n-r-1} \\
    &= \ 3t_{g,n-2}^2 t_{g,n-1} \\
    &< \ 3t_{g,n-1}^2 \\
    &= \ 3t_{g,(n+1)-2}^2.
\end{align*}
In the second step we used the induction hypothesis for the pair $(r,n)$ and in the third step we used (a) and (d), thus (iii) holds for the pair~$(r+1,n+1)$ and it is strict in this case.
\end{proof}

\begin{proof}[Proof of Proposition \ref{prop:fk-minimum}]
We compare the values of the function $f_k$ on the points~$y(g,n,k)$ and $y(g,n,l)$ for $1 \le l \le k$. On $y(g,n,l)$, the value of $f_k$ is given by
\begin{equation*}
    f_k(y(g,n,l)) \ = \ \frac{g}{(n-l+1)^{k-l+1} t_{g,l}^{k-l+2}}.
\end{equation*}
We prove (i). Let $1 \le l \le n$ and set $r := n-l+1$. Then we have
\begin{equation*}
    \frac{f_n(y(g,n,l))}{f_n(y(g,n,n))} \ = \ \frac{f_n(y(g,n,n-r+1))}{f_n(y(g,n,n))} \ = \ \frac{t_{g,n}^2}{r^r t_{g,n-r+1}^{r+1}}.
\end{equation*}
By Lemma \ref{lemma:syl-ineq-technical} (i) this ratio is at least one for all $1 \le r \le n-1$ and equality holds if and only if $r = 1$, ie. if and only if $l = n$. This proves (i).

We prove (ii). Let $k = n-1$. Let $1 \le l \le n-1$ and set $r := n-l$. Then we have
\begin{equation*}
    \frac{f_{n-1}(y(g,n,l))}{f_{n-1}(y(g,n,n-1))} \ = \ \frac{f_{n-1}(y(g,n,n-r))}{f_{n-1}(y(g,n,n-1))} \ = \ \frac{2t_{g,n-1}^2}{(r+1)^r t_{g,n-r}^{r+1}}.
\end{equation*}
Assume $(n,g) \ne (3,1)$. Then this ratio is at least one for all $1 \le r \le n-1$ by Lemma~\ref{lemma:syl-ineq-technical}~(ii). Moreover it is equal to one if and only if $r = 1$ or $(g,r,n) = (1,2,4)$ or~$(g,r,n) = (2,2,3)$. This means that $f_{n-1}$ attains its minimum on $y(g,n,l)$ if and only if $(g,l,n) = (1,2,4)$ or $(g,l,n) = (2,1,3)$ or $l = n-1$. In the case $(n,g) = (3,1)$ we have
\begin{equation*}
    f_2(y(1,3,1)) \ > \ f_2(y(1,3,2)).
\end{equation*}
Thus in this case $f_{2}$ attains its minimum at $y = y(1,3,1)$ and we have $f_{2}(y) = 1/9$.

We prove (iii). Let $k = n-2$. Let $1 \le l \le n-2$ and set $r := n-l-1$. Then we have
\begin{equation*}
    \frac{f_{n-2}(y(g,n,l))}{f_{n-2}(y(g,n,n-2))} \ = \ \frac{f_{n-2}(y(g,n,n-r-1))}{f_{n-2}(y(g,n,n-2))} \ = \ \frac{3t_{g,n-2}^2}{(r+2)^r t_{g,n-r-1}^{r+1}}.
\end{equation*}
Assume $(n,g) \not\in \{(4,1), (4,2), (5,1)\}$. Then by Lemma \ref{lemma:syl-ineq-technical} (iii) this ratio is at least one for all $1 \le r \le n-2$ and it is equal to one if and only if $r = 1$ or~$(g,r,n)$ equals $(3,2,4)$, ie. if and only if $l = n-2$ or $(g,l,n) = (3,1,4)$ holds. For the three cases that were excluded, plugging in the actual values, we obtain
\begin{align*}
    f_2(y(1,4,1)) \ &< \ f_2(y(1,4,2)), \\
    f_2(y(2,4,1)) \ &< \ f_2(y(2,4,2)), \\
    f_3(y(1,5,2)) \ &< \ f_3(y(1,5,1)) \ < \ f_3(y(1,5,3)).
\end{align*}
This completes the proof of Proposition \ref{prop:fk-minimum}.
\end{proof}

\begin{proof}[Proof of Theorem \ref{thm:ufp-bounds}]
Let $A = (\alpha_1,\dots,\alpha_n)$ a unit fraction partition of $g$ and assume~$\alpha_1 \le \dots \le \alpha_n$ holds. Let $y(A) := (1/\alpha_1,\dots,1/\alpha_n).$ This point belongs to $A_g^n$ by Lemma \ref{lemma:ufp_in_Agn}. For $1 \le k \le n$ we have
\begin{equation*}
    F_k(A) \ = \ \frac{\alpha_1 \cdots \alpha_n}{\lcm(\alpha_1, \dots, \alpha_n)^{n-k}} \ \le \ \alpha_1 \cdots \alpha_k \ = \ f_k^{-1}(y(A)).
\end{equation*}
We prove (i). Using Proposition \ref{prop:fk-minimum} (i) for the point $y(A)$ we obtain
\begin{equation*}
    F_n(A) \ \le \ f_n^{-1}(y(A)) \ \le \ \frac{t_{g,n}^2}{g}.
\end{equation*}
Equality holds if and only if $y(A) = y(g,n,n)$, ie. if and only if $A$ is the unit fraction partition $(s_{g,1},\dots,s_{g,n-1},t_{g,n})$.
We prove (ii). We use Proposition \ref{prop:fk-minimum} (ii) for the point $y(A)$. If $(n,g) = (3,1)$ holds, then we have
\begin{equation*}
    F_2(A) \ \le \ f_2^{-1}(y(A)) \ \le \ 9
\end{equation*}
and equality holds if and only if $y(A) = y(1,3,1)$, ie. if and only if $A = (3,3,3)$. If~$(n,g) \ne (3,1)$, then we have
\begin{equation*}
    F_{n-1}(A) \ \le \ f_{n-1}^{-1}(y(A)) \ \le \ \frac{2t_{g,{n-1}}^2}{g}.
\end{equation*}
Checking the points where $f_{n-1}$ attains its minimium, we see that equality holds if and only if $A = (6,6,6)$ or $A = (2,6,6,6)$ or $A = (s_{g,1},\dots,s_{g,n-2},2t_{g,n-1},2t_{g,n-1})$. We prove (iii). We use Proposition \ref{prop:fk-minimum} (iii) for the point $y(A)$ and distinguish three cases:
\begin{enumerate}
    \item[(a)]
    If $n = 4$ and $g \in \{1,2\}$, then we have
    \begin{equation*}
        F_2(A) \ \le \ f_{2}^{-1}(y(A)) \ \le \ 16g^2
    \end{equation*}
    and equality holds if and only if $y(A) = y(g,4,1)$, ie. $A = (4g,4g,4g,4g)$.
    
    \item[(b)] If $(n,g) = (5,1)$, then we have
    \begin{equation*}
        F_3(A) \ \le \ f_{3}^{-1}(y(A)) \ \le \ 128
    \end{equation*}
    and equality holds if and only if $y(A) = y(1,5,2)$, ie. $A = (2,8,8,8,8)$.
    
    \item[(c)]
    If $(n,g) \not\in \{ (4,1), (4,2), (5,1) \}$, then
    \begin{equation*}
        F_{n-2}(A) \ \le \ f_{n-2}^{-1}(y(A)) \ \le \ \frac{3t_{g,n-2}^2}{g}.
    \end{equation*}
    Equality holds if and only if $y(A) = y(3,4,1)$ or $y(A) = y(g,n,n-2)$, ie. if and only if $A$ is one of
    \begin{equation*}
    \qquad
        (12,12,12,12),
        \qquad
        (s_{g,1},\dots,s_{g,n-3},3t_{g,n-2},3t_{g,n-2},3t_{g,n-2}).
    \end{equation*}
\end{enumerate}
\end{proof}

\goodbreak

\section{Proofs of Theorems \ref{thm:max-vol} - \ref{thm:max-mahler}}\label{sect:proofs}

We prove Theorems \ref{thm:max-vol} - \ref{thm:max-mahler}. We start with three Lemmas that we need for the proof of Theorem \ref{thm:max-vol}.

\begin{lemma}\label{lemma:lambdastar-denom}
For any $d$-dimensional Fano simplex $\Delta$ of Gorenstein index $g$ the product~$g^{d-1} \lambda(\Delta^*)$ is an integer.
\end{lemma}

\begin{proof}
Let $\Delta$ a $d$-dimensional Fano simplex of Gorenstein index $g$. For the weight system of its dual we write $Q(\Delta^*) = (q^*_0,\dots,q^*_d)$. We show that $g^{d-1} q^*_i$ is an integer for all~$i = 0,\dots,d$. Denote by $v_0,\dots,v_d$ the vertices of $\Delta$ and by $u_0,\dots,u_d$ the vertices of~$\Delta^*$, ordered in such a way that $\bangle{u_i,v_j} = -1$ holds whenever $i \ne j$. We have
\begin{equation*}
    q^*_i \ = \ |\det(u_0,\dots,\hat{u}_i,\dots,u_d)|.
\end{equation*}
Let $i \in \{0,\dots,d\}$ and extend $(v_i)$ to a basis $(v_i = b_1, b_2,\dots,b_d)$ of $\ZZ^{d}$. Denote by~$C = (c_1,\dots,c_d)$ the dual basis. For $j \ne i$ we write $u_j$ as a linear combination of the basis $C$ with coefficients~$\mu_{j1},\dots,\mu_{jd} \in \frac{1}{g}\ZZ$. We have
\begin{equation*}
    \mu_{j1} \ = \ \sum\limits_{k=1}^d \mu_{jk} \bangle{c_k,b_1} \ = \ \bangle{u_j,v_i} \ = \ -1.
\end{equation*}
Using this presentation of $u_j$ with respect to the basis $C = (c_1,\dots,c_d)$, we obtain for $q^*_i$:
\begin{equation*}
    q^*_i
    \ = \
    |\det(u_0,\dots,\hat{u}_i,\dots,u_d)|
    \ = \
    |\det
    \left[\begin{array}{ccc}
        -1 & \dots & -1 \\
        \mu_{02} & \dots & \mu_{d2} \\
        \vdots & & \vdots \\
        \mu_{0d} & \dots & \mu_{dd}
    \end{array}
    \right]|
    \ = \
    \frac{K}{g^{d-1}}
\end{equation*}
for some $K \in \ZZ_{>0}$. Thus $g^{d-1}Q(\Delta^*)$ is an integral weight system, which shows that~$g^{d-1} \lambda(\Delta^*)$ is an integer.
\end{proof}

\begin{lemma}\label{lemma:fano-simp-Pmat}
Let $\Delta$ a Fano simplex of dimension $d \ge 2$ and Gorenstein index~$g$. Let $A(\Delta) = (\alpha_0,\dots,\alpha_d)$ the associated unit fraction partition. Write $\Delta \cong \Delta(P)$, where
\begin{equation*}
    P
    \ := \ 
    \left[\begin{array}{ccccc}
        1 & a_{12} & \cdots & a_{1d} & -b_1\\
        0 & a_{22} & \cdots & a_{2d} & -b_2 \\
        \vdots & \ddots & \ddots & \vdots & \vdots \\
        0 & \cdots & 0 & a_{dd} & -b_d \\
    \end{array}\right]
\end{equation*}
is in Hermite normal form. Then the following hold:
\begin{enumerate}
    \item $a_{kk}$ divides $\alpha_{k-1}$ for all $k = 2,\dots,d$.
    \item $a_{22}$ divides $\alpha_{0}$.
    \item If $\gcd(\alpha_i,\alpha_{k-1}) = 1$ holds for all $i=0,\dots,k-2$, then we have $a_{kk} = 1$.
\end{enumerate}
\end{lemma}

\begin{proof}
For item (i) we refer to the proof of Proposition \ref{prop:simp-Pmat-finite}. We prove~(ii). Denote the columns of $P$ by $v_0,\dots,v_d$. For the first and the last Gorenstein form of~$\Delta$ we have
\begin{equation*}
    u_0 \ = \ \left( \frac{\alpha_0}{g}-1, \frac{a_{12}-1}{a_{22}} - \frac{a_{12} \alpha_0}{a_{22}g}, u_{03}, \dots, u_{0d} \right) \ \in \ \frac{1}{g}\ZZ,
\end{equation*}
\begin{equation*}
    u_d \ = \ \left( -1, \frac{a_{12}-1}{a_{22}}, u_{d3}, \dots, u_{dd} \right) \ \in \ \frac{1}{g}\ZZ.
\end{equation*}
Taking their difference, we see that $a_{12}\alpha_0/a_{22}$ must be an integer. Since $a_{12}$ and $a_{22}$ are coprime, this means that $a_{22}$ divides $\alpha_0$. We prove (iii). Assume that~$\gcd(\alpha_0,\dots,\alpha_k) = 1$ holds. We show by induction on $l$ that $a_{ll} = 1$ holds for all $l \le k$. For $l = 1$ there is nothing to prove. Let $l = 2$. By item (i), $a_{22}$ divides~$\alpha_1$ and by item (ii), $a_{22}$ divides $\alpha_0$. As they are coprime, we obtain $a_{22} = 1$. Now assume $l > 2$ and $a_{ii} = 1$ for all $i < l$. Then the $i$th Gorenstein form for $i < l$ and the last Gorenstein form of $\Delta$ are given by
\begin{equation*}
    u_i \ = \ \left( -1, \dots, \frac{\alpha_{i-1}}{g} - 1, \dots, -1, u_{il}, \dots, u_{id} \right) \ \in \ \frac{1}{g}\ZZ,
\end{equation*}
\begin{equation*}
    u_d \ = \ \left( -1, \dots, -1, u_{dl}, \dots, u_{dd} \right) \ \in \ \frac{1}{g}\ZZ,
\end{equation*}
where the entry $\alpha_{i-1}/g-1$ of $u_i$ is at the $i$th position. Evaluating their difference on the vector $v_{l-1} = (a_{1l},\dots,a_{ll},0,\dots,0)$ shows that $a_{ll}$ divides $\alpha_{i-1}a_{il}$. Since~$a_{ll}$ divides $\alpha_{k-1}$ by item (i), it is coprime to $\alpha_{i-1}$. Thus $a_{ll}$ divides $a_{il}$. This is only possible if $a_{il} = 0$. Now, the column $v_{l-1}$ is a primitive point in $\ZZ^d$. This yields~$a_{ll} = 1$.
\end{proof}

\begin{proposition}\label{prop:w-not-(1,1,1)}
Let $\Delta$ a Fano triangle of Gorenstein index $g$. If $Q_\Delta^\red = (1,1,1)$ holds then $g$ is odd.
\end{proposition}

\begin{proof}
Let $\Delta$ a Fano triangle with even Gorenstein index $g$ and assume that~$Q_\Delta^\red~=~(1,1,1)$ holds. The unit fraction partition of $g$ associated with~$\Delta$ is
\begin{equation*}
    A(\Delta) \ = \ (\alpha_0,\alpha_1,\alpha_2) \ = \ (3g, 3g, 3g).
\end{equation*}
Let $P \in \Mat(2,3;\ZZ)$ such that $\Delta(P) \cong \Delta$ holds. We may write
\begin{equation*}
    P \ = \
    \left[\begin{array}{rrr}
        1 & a & -(a+1)\\
        0 & b & -b
    \end{array}\right]
\end{equation*}
for some non-negative $a,b \in \ZZ$. Note that for the columns of $P$ to all be primitive,~$b$ must be odd. The Gorenstein forms $u_0,u_1,u_2$ of $\Delta$ are given by
\begin{equation*}
    u_0 \ = \ \left( 2 , - \frac{2a+1}{b} \right), \qquad u_1 \ = \ \left( -1 , \frac{a+2}{b} \right), \qquad u_2 \ = \ \left( -1 , \frac{a-1}{b} \right).
\end{equation*}
Thus the local Gorenstein indices $g_0,g_1,g_2$ of $\Delta$ all divide $b$. In particular, the Gorenstein index $g = \lcm(g_0,g_1,g_2)$ divides $b$. Since $g$ is even, this contradicts the fact that $b$ is odd. Thus $Q_\Delta^\red$ cannot be equal to $(1,1,1)$.
\end{proof}

\begin{proof}[Proof of Theorem \ref{thm:max-vol}]
Let $\Delta$ a $d$-dimensional IP lattice simplex of Gorenstein index~$g$ and associated unit fraction partition $A(\Delta) = (\alpha_0,\dots,\alpha_d)$. We may assume that~$\alpha_0~\le~\dots~\le~\alpha_d$ holds. By Lemma \ref{lemma:lambdastar-denom}, the product $g^{d-1}\lambda(\Delta^*)$ is an integer. In particular $g^{d-1}\lambda(\Delta^*) \ge 1$ holds. With Proposition \ref{prop:uf-vol} (ii) we obtain the following volume bound for $\Delta$:
\begin{equation}\label{eq:vol-ineq}
    \Vol(\Delta) \ \le \ g^{d-1} \lambda(\Delta^*)\Vol(\Delta) \ = \ \frac{1}{g}\frac{\alpha_0 \cdots \alpha_d}{\lcm(\alpha_0,\dots,\alpha_d)}.
\end{equation}
Equality holds if and only if $\lambda(\Delta^*) = 1/g^{d-1}$, which by Proposition \ref{prop:uf-vol} (iv) is equivalent to 
\begin{equation*}
    \lambda(\Delta) \ = \ \frac{\alpha_0 \cdots \alpha_d}{ \lcm(\alpha_0,\dots,\alpha_d)^2 }.
\end{equation*}
We use Theorem \ref{thm:ufp-bounds} (ii) to bound the right hand side of Equation \ref{eq:vol-ineq} from above. In case $(d,g) = (2,1)$ we have
\begin{equation*}
    \Vol(\Delta)
    \ \le \
    \frac{1}{g}\frac{\alpha_0 \cdots \alpha_d}{\lcm(\alpha_0,\dots,\alpha_d)}
    \ = \
    \frac{\alpha_0 \cdots \alpha_d}{\lcm(\alpha_0,\dots,\alpha_d)}
    \ \le \
    9.
\end{equation*}
If equality holds, then we have $A(\Delta) = (3,3,3)$, ie. $Q_\Delta^\red = (1,1,1)$. Thus~$\Delta$ is isomorphic to $H \Delta(1,1,1)$ for some $2\times 2$ integer matrix $H$ with $\det(H) = \lambda(\Delta) = 3$. We may assume that $H$ is in Hermite normal form, thus we have $\Delta \cong \Delta(P)$ with
\begin{equation*}
    P
    \ = \
    \left[\begin{array}{ccc}
        1 & a & -(a+1) \\
        0 & 3 & -3
    \end{array}\right],
\end{equation*}
for $a \in \{1,2\}$. The two choices of $a$ lead to isomorphic simplices. We may choose~$a = 1$, which yields
\begin{equation*}
    P
    \ = \
    \left[\begin{array}{ccc}
        1 & 1 & -2 \\
        0 & 3 & -3
    \end{array}\right].
\end{equation*}
Now assume $(d,g) \ne (2,1)$ holds. Then by Theorem \ref{thm:ufp-bounds} (ii) we have
\begin{equation*}
    \Vol(\Delta)
    \ \le \
    \frac{1}{g}\frac{\alpha_0 \cdots \alpha_d}{\lcm(\alpha_0,\dots,\alpha_d)}
    \ \le \
    \frac{2t_{g,d}^2}{g^2}.
\end{equation*}
If equality holds, then we have $A(\Delta) = A$, where $A$ is one of the following:
\begin{equation*}
    A \ = \ (6,6,6), \qquad A \ = \ (2,6,6,6), \qquad A \ = \ (s_{g,1},\dots,s_{g,d-1},2t_{g,d},2t_{g,d}). 
\end{equation*}
Note that by Proposition \ref{prop:w-not-(1,1,1)} there is no Fano simplex $\Delta$ with associated unit fraction partition $A(\Delta) = (6,6,6)$. The other two cases give the following reduced weight systems
\begin{equation*}
    Q \ = \ (3,1,1,1),
    \qquad\qquad
    Q \ = \ \left( \frac{2t_{g,d}}{s_{g,1}},\dots,\frac{2t_{g,d}}{s_{g,d-1}}, 1, 1 \right)
\end{equation*}
and $\Delta$ is isomorphic to $H \Delta(Q)$, where $Q$ is one of the reduced weight systems above and~$H$ is a square integer matrix in Hermite normal form with $\det(H) = \lambda(\Delta)$. We first consider the case
\begin{equation*}
    A(\Delta) \ = \ (2,6,6,6),
    \qquad
    g(\Delta) \ = \ 1,
    \qquad
    Q_\Delta^\red = (3,1,1,1),
    \qquad
    \lambda(\Delta) \ = \ 12.
\end{equation*}
We consider the diagonal entries $(a_{11},a_{22},a_{33})$. Since $\Delta$ is a Fano simplex we have~$a_{11} = 1$. Moreover, Propositions \ref{prop:simp-Pmat-constr} and \ref{prop:simp-Pmat-finite} tell us that $a_{22}$ and $a_{33}$ are both divisors of $6$. As we have
\begin{equation*}
    a_{22}\cdot a_{33} \ = \ \det(H) \ = \ \lambda(\Delta) \ = \ 12,
\end{equation*}
this leaves for the diagonal of $H$ only the two possibilities $(a_{11},a_{22},a_{33}) = (1,2,6)$ and~$(a_{11},a_{22},a_{33}) = (1,6,2)$. The second case can be transformed into the first by changing the second and third column of $H$ and bringing it in Hermite normal form again. Thus there are $0 \le a,b < 5$ such that $\Delta \cong \Delta(P)$, where
\begin{equation*}
    P
    \ = \
    \left[\begin{array}{rrrr}
        1 & 1 & a & -(4+a)\\
        0 & 2 & b & -(2+b)\\
        0 & 0 & 6 & -6
    \end{array}\right].
\end{equation*}
The Gorenstein forms of $\Delta$ are then given by
\begin{equation*}
    u_0 \ = \ \left( 1, -1, \frac{b-a-1}{6} \right),
    \qquad
    u_1 \ = \ \left( -1, 3, \frac{a-3b-1}{6} \right),
\end{equation*}
\begin{equation*}
    u_2 \ = \ \left( -1, 0, \frac{5+a}{6} \right),
    \qquad
    u_3 \ = \ \left( -1, 0, \frac{a-1}{6} \right).
\end{equation*}
Since $\Delta$ is of Gorenstein index $1$, all its Gorenstein forms are integral. The last entry of $u_3$ thus dictates $a = 1$. Plugging this into $P$, we obtain $u_1 = (-1,3,b/2)$. We obtain~$b = 2$ and $P$ is the first matrix from Theorem \ref{thm:max-vol} (ii). Now consider the case
\begin{equation*}
    A(\Delta) \ = \ (s_{g,1},\dots,s_{g,d-1},2t_{g,d},2t_{g,d}),
    \qquad
    g(\Delta) \ = \ g,
\end{equation*}
\begin{equation*}
    Q_\Delta^\red \ = \ \left( \frac{2t_{g,d}}{s_{g,1}},\dots,\frac{2t_{g,d}}{s_{g,d-1}}, 1, 1 \right),
    \qquad
    \lambda(\Delta) \ = \ \frac{t_{g,d}}{g}.
\end{equation*}
We have $\Delta \cong H\cdot \Delta(Q_\Delta^\red) = \Delta(P)$, where the matrices $H$ and $P$ are given by
\begin{equation*}
    H
    \ = \
    \left[\begin{array}{cccc}
        1 & a_{12} & \cdots & a_{1d} \\
        0 & a_{22} & \cdots & a_{2d} \\
        \vdots & \ddots & \ddots & \vdots \\
        0 & \cdots & 0 & a_{dd} \\
    \end{array}\right],
    \qquad
    P
    \ = \
    \left[\begin{array}{ccccc}
        1 & a_{12} & \cdots & a_{1d} & -b_1\\
        0 & a_{22} & \cdots & a_{2d} & -b_2 \\
        \vdots & \ddots & \ddots & \vdots & \vdots \\
        0 & \cdots & 0 & a_{dd} & -b_d \\
    \end{array}\right].
\end{equation*}
The entries $b_1,\dots,b_d$ in the last column of $P$ can be expressed via the $a_{ij}$ by solving the linear system~$P\cdot Q_\Delta^\red = 0$. The determinant of $H$ satisfies
\begin{equation*}
    \det(H) \ = \ \lambda(\Delta) \ = \ \frac{t_{g,d}}{g} \ = \ s_{g,1} \cdots s_{g,d-1}. 
\end{equation*}
We first consider the case $d = 2$. Then $a_{22} = s_{g,1} = g+1$ holds. The last Gorenstein form of $\Delta$ reads
\begin{equation*}
    u_2 \ = \ \left( -1, \frac{a_{12}-1}{g+1} \right)
\end{equation*}
and we have $0 \le a_{12} < g+1$. The entries of $u_2$ are in $\frac{1}{g}\ZZ$. Thus $a_{12}-1$ is a multiple of~$g+1$. This is only possible for $a_{12} = 1$, which yields
\begin{equation*}
    P
    \ = \
    \left[\begin{array}{rrr}
        1 & 1 & -(2g+1)\\
        0 & g+1 & -(g+1)
    \end{array}\right]
    \ = \
    \left[\begin{array}{rrr}
        1 & \frac{(s_{g,1}-g)}{s_{g,1}}\frac{t_{g,2}}{g} & -\frac{(s_{g,1}+g)}{s_{g,1}}\frac{t_{g,2}}{g}\\
        0 & \frac{t_{g,2}}{g} & -\frac{t_{g,2}}{g}
    \end{array}\right].
\end{equation*}
Now assume $d > 2$. Note that the entries $\alpha_0,\dots,\alpha_{d-2}$ of the ufp $A(\Delta)$ are pairwise coprime. By Lemma \ref{lemma:fano-simp-Pmat} (iii) we have $a_{kk} = 1$ for all $k = 2,\dots,d-1$. Moreover we obtain $a_{dd} = \det(H) = s_{g,1} \cdots s_{g,d-1}$. We now show that $a_{kd} = \frac{(s_{g,k}-g)}{s_{g,k}}\frac{t_{g,d}}{g}$ holds for all~$1 \le k \le d-1$.  We set $m := a_{1d} + \dots + a_{(d-1)d} - 1$. The Gorenstein forms of $\Delta$ are given by
\begin{equation*}
    u_{k-1} \ = \ \left( -1, \dots, \frac{s_{g,k}}{g}-1, \dots, -1, \frac{m}{a_{dd}} - \frac{s_{g,k}\,a_{kd}}{g\,a_{dd}} \right) \ \in \ \frac{1}{g} \ZZ,
\end{equation*}
\begin{equation*}
    u_d \ = \ \left( -1, \dots, -1, \frac{m}{a_{dd}} \right) \ \in \ \frac{1}{g} \ZZ,
\end{equation*}
where $k = 1,\dots,d$ and the entry $s_{g,k}/g-1$ of $u_{k-1}$ is at the $k$th position. Note that~$a_{dd}$ is coprime to $g$. The last entry of $u_d$ thus dictates that $a_{dd}$ divides $m$. Moreover, by the last entry of $u_{k-1}$, $a_{kd}$ is a multiple of $s_{g,1} \cdots \hat{s}_{g,k} \cdots s_{g,d-1}$, where~$\hat{s}_{g,k}$ means, that $s_{g,k}$ is omitted in the product. There is thus $\Lambda_k \in \ZZ$ with
\begin{equation*}
    a_{kd} \ = \ \Lambda_k \, s_{g,1} \cdots \hat{s}_{g,k} \cdots s_{g,d-1}.
\end{equation*}
Using these $\Lambda_k$, we can write the integer $m$ as
\begin{equation*}
    m \ = \ \frac{a_{1d} + \dots + a_{(d-1)d} - 1}{a_{dd}} \ = \ \frac{\Lambda_1}{s_{g,1}} + \dots + \frac{\Lambda_{d-1}}{s_{g,d-1}} - \frac{g}{t_{g,d}}.
\end{equation*}
We now treat the $\Lambda_1,\dots,\Lambda_{d-1}$ as indeterminates. Note that they only appear in the last entry of the Gorenstein forms $u_0,\dots,u_{d}$, whereas they appear only in the first~$d-1$ entries of the last column $v_{d}$ of $P$. Evaluating $u_0,\dots,u_{d-2}$ on $v_d$ thus gives a system of~$d-1$ linear equations in the $d-1$ variables $\mu_1,\dots,\mu_{d-1}$, which are independent since the Gorenstein forms $u_0,\dots,u_{d-2}$ are linearly independent. This system thus has at most one solution. A direct computation shows that the choice~$\Lambda_k = s_{g,k} - g$ is a solution for that system. This shows that $P$ is the second matrix in Theorem \ref{thm:max-vol} (ii), which completes the proof of the Theorem.
\end{proof}

\begin{proof}[Proof of Theorem \ref{thm:max-mahler}]
Let $\Delta$ a $d$-dimensional IP simplex of Gorenstein index~$g$ and associated unit fraction partition $A(\Delta) = (\alpha_0,\dots,\alpha_d)$. We may assume that $A$ is ordered, ie. that~$\alpha_0 \le \dots \le \alpha_d$ holds. By Proposition \ref{prop:uf-vol} (i) we have
\begin{equation*}
    \Vol(\Delta)\Vol(\Delta^*) \ = \ \frac{\alpha_0 \cdots \alpha_d}{g^{d+1}}.
\end{equation*}
By Theorem \ref{thm:ufp-bounds} (i) the numerator of the right hand side is bounded by $t_{g,d+1}^2/g$. Thus we obtain
\begin{equation*}
    \Vol(\Delta)\Vol(\Delta^*) \ \le \ \frac{t_{g,d+1}^2}{g^{d+2}}.
\end{equation*}
If equality holds, then by Theorem \ref{thm:ufp-bounds} (i) we have $A(\Delta) = (s_{g,1},\dots,s_{g,d},t_{g,d+1})$, which is equivalent to
\begin{equation*}
    Q_\Delta^\red \ = \ Q(A(\Delta)) \ = \ \left( \frac{t_{g,d+1}}{s_{g,1}}, \dots, \frac{t_{g,d+1}}{s_{g,d}}, 1 \right).
\end{equation*}
On the other hand, assume that $Q_\Delta^\red$ is of this form. Let $\Delta' = \Delta(Q_\Delta^\red)$. There is~$H \in \GL(d+1,\QQ)$ such that $\Delta \cong H\cdot \Delta'$ holds. For the Mahler volume of $\Delta$ we thus obtain
\begin{equation*}
    \Vol(\Delta) \Vol(\Delta^*)
    \ = \
    \Vol( H \Delta' ) \, \Vol( (H^*)^{-1} (\Delta')^* )
    \ = \
    \Vol( \Delta' ) \Vol( (\Delta')^* )
    \ = \
    \frac{t_{g,d+1}^2}{g^{d+2}}.
\end{equation*}
\end{proof}

\begin{proof}[Proof of Theorem \ref{thm:fwps-max-mult}]
Let $\Delta$ a $d$-dimensional Fano simplex of Gorenstein index $g$ and associated unit fraction partition $A(\Delta) = (\alpha_0,\dots,\alpha_d)$. We may assume that the entries of $A(\Delta)$ satisfy $\alpha_0 \le \dots \le \alpha_d$. By Proposition \ref{prop:uf-vol} (iv) and Lemma \ref{lemma:lambdastar-denom} we have
\begin{equation}\label{eq:fact-bound}
    \lambda(\Delta)
    \ \le \
    g^{d-1} \lambda(\Delta) \lambda(\Delta^*)
    \ = \
    \frac{\alpha_0\cdots \alpha_d}{\lcm(\alpha_0,\dots,\alpha_d)^2}.
\end{equation}
We prove (i). Let $d = 3$ and $g \in \{1,2\}$. By Theorem \ref{thm:ufp-bounds} (iii) the right hand side of Equation \ref{eq:fact-bound} is bounded by $16g^2$. Assume $\lambda(\Delta) = 16g$ holds. Then we have
\begin{equation*}
    A(\Delta) \ = \ (4g,4g,4g,4g), \qquad Q_\Delta^\red \ = \ (1,1,1,1)
\end{equation*}
and there is a $3 \times 3$ integer matrix $H$ in Hermite normal form with determinant equal to~$\lambda(\Delta) = 16g^2$, such that $\Delta \cong H \cdot \Delta(1,1,1,1)$ holds. Thus we can write~$\Delta \cong \Delta(P)$ with
\begin{equation*}
    P
    \ = \ 
    \left[\begin{array}{rrrr}
        1 & a_{12} & a_{13} & -(a_{12}+a_{13}+1) \\
        0 & a_{22} & a_{23} & -(a_{22}+a_{23}) \\
        0 & 0      & a_{33} & -a_{33}
    \end{array}\right]
\end{equation*}
in Hermite normal form. By Lemma \ref{lemma:fano-simp-Pmat} (i), $a_{22}$ and $a_{33}$ each divide~$4g$, moreover we have $a_{22}\cdot a_{33} = \det(H) = 16g^2$. Thus $a_{22} = a_{33} = 4g$ holds. The difference of the Gorenstein forms $u_1$ and $u_2$ of $\Delta$ is given by
\begin{equation*}
   u_1 - u_2 \ = \ \left( 0, \frac{1}{g}, - \frac{a_{23}+4g}{4g^2} \right) \ \in \ \frac{1}{g}\ZZ.
\end{equation*}
Thus $4g$ divides $a_{23}$. This is only possible for $a_{23} = 0$. The last Gorenstein form of~$\Delta$ then reads
\begin{equation*}
    u_{3} \ = \ \left( -1, \frac{a_{12}-1}{4g}, \frac{a_{13}-1}{4g} \right),
\end{equation*}
which yields $a_{12} = 4k+1$ and $a_{13} = 4l+1$ for some $k,l \in \ZZ$. Taking the restrictions on~$a_{12}$ and $a_{13}$ into account we obtain $0 \le k,l \le g-1$. In case $g = 1$ we have~$k = 0$ and~$l = 0$. In case $g = 2$ the different choices for $k$ and $l$ lead to isomorphic simplices. We may thus assume $k = l = 0$ and $P$ is of the form stated in Theorem~\ref{thm:fwps-max-mult} (i).

We prove (ii). Let $(d,g) = (4,1)$. By Equation 
\ref{eq:fact-bound} and Theorem \ref{thm:ufp-bounds} (iii) we have $\lambda(\Delta) \le 128$. If equality holds, then we have
\begin{equation*}
    A(\Delta) \ = \ (2,8,8,8,8), \qquad Q_\Delta^\red \ = \ (4,1,1,1,1), \qquad \lambda(\Delta) \ = \ 128.
\end{equation*}
There is a $4 \times 4$ integer matrix $H$ in Hermite normal form with determinant equal to~$\lambda(\Delta) = 128$, such that $\Delta \cong H \cdot \Delta(4,1,1,1,1)$ holds. Thus $\Delta \cong \Delta(P)$ holds with
\begin{equation*}
    P
    \ = \ 
    \left[\begin{array}{rrrrr}
        1 & a_{12} & a_{13} & a_{14} & -(4+a_{12}+a_{13}+a_{14}) \\
        0 & a_{22} & a_{23} & a_{24} & -(a_{22}+a_{23}+a_{24}) \\
        0 & 0      & a_{33} & a_{34} & -(a_{33}+a_{34}) \\
        0 & 0      & 0      & a_{44} & -a_{44}
    \end{array}\right]
\end{equation*}
in Hermite normal form. By Lemma \ref{lemma:fano-simp-Pmat} we have $a_{22} \mid 2$ and $a_{33},a_{44} \mid 8$. Moreover the product of the diagonal entries is the determinant of $H$. The only possibility for the diagonal is thus $(a_{22},a_{33},a_{44}) = (2,8,8)$. Calculating the Gorenstein forms of $\Delta$ and using the fact that $\Delta$ is of Gorenstein index $1$, we obtain
\begin{equation*}
    a_{12} \ = \ a_{13} \ = \ a_{14} \ = \ 1, \qquad a_{23} \ = \ a_{24} \ = \ 2, \qquad a_{34} \ = \ 8.
\end{equation*}
This shows that $P$ is the matrix from Theorem \ref{thm:fwps-max-mult} (ii).

We prove (iii). Assume that $(d,g)$ is neither of $(3,1),$ $(3,2),$ $(4,1)$. By Equation~\ref{eq:fact-bound} and Theorem \ref{thm:ufp-bounds} (iii)-(c) we have
\begin{equation*}
    \lambda(\Delta) \ \le \ \frac{3 t_{g,d-1}^2}{g}.
\end{equation*}
If equality holds, then we have $A(\Delta) = A$, where $A$ is one of the following unit fraction partitions:
\begin{equation*}
    A \ = \ (12,12,12,12), \qquad A \ = \ (s_{g,1},\dots,s_{g,d-2},3t_{g,d-1},3t_{g,d-1},3t_{g,d-1}).
\end{equation*}
In the first case we are in the situation $(d,g) = (3,3)$ and we have
\begin{equation*}
    A(\Delta) \ = \ (12,12,12,12), \qquad Q_{\Delta}^\red \ = \ (1,1,1,1), \qquad \lambda(\Delta) \ = \ 144.
\end{equation*}
Again, we have $\Delta \cong \Delta(P)$ with
\begin{equation*}
    P
    \ = \ 
    \left[\begin{array}{rrrr}
        1 & a_{12} & a_{13} & -(1+a_{12}+a_{13}) \\
        0 & a_{22} & a_{23} & -(a_{22}+a_{23}) \\
        0 & 0      & a_{33} & -a_{33}
    \end{array}\right]
\end{equation*}
in Hermite normal form. By Lemma \ref{lemma:fano-simp-Pmat} both $a_{22}$ and $a_{33}$ are divisors of 12. Moreover we have $a_{22} \cdot a_{33} = \lambda(\Delta) = 144$. Thus $a_{22} = a_{33} = 12$ holds. Calculating the Gorenstein forms of $\Delta$ and using the fact that $\Delta$ is of Gorenstein index $3$, we obtain
\begin{equation*}
    a_{23} \ = \ 0, \qquad a_{12} \ = \ 4k+1, \qquad a_{13} \ = \ 4l+1,
\end{equation*}
where $0 \le k,l \le 2$. The cases $k = 2$ and $l = 2$, as well as $(k,l) = (0,0)$ lead to a non-primitive column of $P$. Thus these cases are excluded. All other choices for $k,l$ lead to isomorphic matrices. We may thus choose $(k,l) = (0,1)$ and $P$ is the first matrix from Theorem \ref{thm:fwps-max-mult} (iii). We now consider the second possible unit fraction partition for $\Delta$, ie. we have $(d,g) \ne (3,3)$ and
\begin{equation*}
    A(\Delta) \ = \ (s_{g,1},\dots,s_{g,d-2},3t_{g,d-1},3t_{g,d-1},3t_{g,d-1}),
\end{equation*}
\begin{equation*}
    Q_\Delta^\red \ = \ \left( \frac{3t_{g,d-1}}{s_{g,1}}, \dots, \frac{3t_{g,d-1}}{s_{g,1}}, 1, 1, 1 \right),
    \qquad
    \lambda(\Delta) \ = \ \frac{3 t_{g,d-1}^2}{g}.
\end{equation*}
As we have done before, we can write $\Delta \cong \Delta(P)$, where
\begin{equation*}
    P
    \ = \
    \left[\begin{array}{ccccc}
        1 & a_{12} & \cdots & a_{1d} & -b_1\\
        0 & a_{22} & \cdots & a_{2d} & -b_2 \\
        \vdots & \ddots & \ddots & \vdots & \vdots \\
        0 & \cdots & 0 & a_{dd} & -b_d \\
    \end{array}\right]
\end{equation*}
is in Hermite normal form. The entries $b_k$ of the last column of $P$ can be computed from the entries $a_{k,j}$ by solving the linear system $P\cdot Q_\Delta^\red = 0$. Moreover we have
\begin{equation*}
    a_{22} \cdots a_{dd} \ = \ \lambda(\Delta) \ = \ \frac{3t_{g,d-1}^2}{g} \ = \ 3 t_{g,d-1} \cdot s_{g,1} \cdots s_{g,d-2}.
\end{equation*}
Note that the entries $\alpha_{0}, \dots, \alpha_{d-3}$ of $A(\Delta)$ are pairwise coprime. Thus for the diagonal entries of $P$ Lemma \ref{lemma:fano-simp-Pmat} (iii) yields~$a_{22} = \dots = a_{(d-2)(d-2)} = 1$. Moreover both~$a_{(d-1)(d-1)}$ and $a_{dd}$ are divisors of $3t_{g,d-1}$ by Lemma \ref{lemma:fano-simp-Pmat} (i). Comparing this to the product of the diagonal entries we obtain that both $a_{(d-1)(d-1)}$ and $a_{dd}$ are multiples of the product~$s_{g,1} \cdots s_{g,d-2}$ and that
\begin{equation*}
    a_{(d-1)(d-1)}\cdot a_{dd} \ = \ s_{g,1} \cdots s_{g,d-2} \cdot 3g \cdot s_{g,1} \cdots s_{g,d-2}
\end{equation*}
holds.  We can thus write $a_{(d-1)(d-1)} = \Lambda s_{g,1} \cdots s_{g,d-2}$ for some divisor $\Delta$ of $3g$. We show that $\Lambda = 1$ holds. Set $m := a_{1(d-1)} + \dots + a_{(d-2)(d-1)} - 1$. The Gorenstein forms $u_d$ as well as $u_{k}$ for~$k = 0,\dots, d-3$ of $\Delta$ are given by
\begin{equation*}
    u_k \ = \ \left(-1, \dots, \frac{\alpha_k}{g}-1, \dots, -1, \frac{m}{a_{(d-1)(d-1)}} - \frac{\alpha_{k}a_{(k+1)(d-1)}}{ga_{(d-1)(d-1)}}, u_{kd}\right) \ \in \frac{1}{g} \ZZ,
\end{equation*}
\begin{equation*}
    u_d \ = \ \left( -1, \dots, -1, \frac{m}{a_{(d-1)(d-1)}}, u_{dd} \right) \ \in \frac{1}{g} \ZZ.
\end{equation*}
Here the entry $\alpha_k/g-1$ is at the position $k+1$ of $u_{k}$. Let $1 \le k \le d-2$. Comparing the second to last entries of $u_{k-1}$ and $u_d$ and using the fact that they are in $\frac{1}{g}\ZZ$, we must have that $a_{(d-1)(d-1)}$ divides $\alpha_{k-1} a_{k(d-1)} = s_{g,k}a_{k(d-1)}$. Thus we can write
\begin{equation*}
    a_{k(d-1)} \ = \ \Lambda_k s_{g,1} \cdots \hat{s}_{g,k} \cdots s_{g,d-2}
\end{equation*}
for some $\Lambda_{k} \in \ZZ_{\ge 1}$. Here $\hat{s}_{g,k}$ means that $s_{g,k}$ is omitted. Note that since $s_{g,k} a_{k(d-1)}$ is a multiple of $a_{(d-1)(d-1)}$, the number $\Lambda_k$ is a multiple of $\Lambda$. So in order for the column~$v_{d-2}$ to be primitive, $\Lambda$ must be equal to $1$. We thus have
\begin{equation*}
    a_{(d-1)(d-1)} \ = \ s_{g,1} \cdots s_{g,d-2}, \qquad a_{dd} \ = \ 3 t_{g,d-1}.
\end{equation*}
It remains show that $\Lambda_k = (s_{g,k}-g)$ holds for all $k = 1,\dots,d-2$. The situation is very similar to the last part of the proof of Theorem \ref{thm:max-vol}. However, as we do not have information about the entries $a_{1d},\dots,a_{(d-1)d}$ of $P$, we need to employ a different strategy. Note that $a_{(d-1)(d-1)}$ is coprime to $g$. Considering again the last Gorenstein form $u_d$ of~$\Delta$, its entry
\begin{equation*}
    u_{d(d-1)} \ = \ \frac{m}{a_{(d-1)(d-1)}} \ = \ \frac{a_{1(d-1)} + \dots + a_{(d-2)(d-1)} - 1}{a_{(d-1)(d-1)}}
\end{equation*}
must be an integer. In particular, $s_{g,k}$ divides $m$ for all $k = 1,\dots,d-2$. Since $a_{l(d-1)}$ is a multiple of $s_{g,k}$ for $l \ne k$, this means that we have
\begin{equation*}
    s_{g,k} \ \mid \ \Lambda_k s_{g,1} \cdots \hat{s}_{g,k} \cdots s_{g,d-2} - 1.
\end{equation*}
As $s_{g,l} = t_{g,l} + 1$ and $s_{g,k} \mid t_{g,l}$ holds for $l > k$, this implies that we have
\begin{equation*}
    s_{g,k} \ \mid \ s_{g,1} \cdots s_{g,k-1} \Lambda_k - 1.
\end{equation*}
Thus there is $B \in \ZZ$ with $B s_{g,k} \ = \ s_{g,1} \cdots s_{g,k-1} \Lambda_k - 1$. Since $P$ is in Hermite normal form, $\Lambda_k$ is in the range $0 \le \Lambda_k < s_{g,k}$. Thus $B$ is at least one, but less than~$s_{g,1} \cdots s_{g,k-1}$. Moreover, we obtain the identity
\begin{equation*}
    s_{g,1} \cdots s_{g,k-1} \Lambda_k \ = \ B s_{g,k} + 1 \ = \ B( t_{g,k} + 1 ) + 1,
\end{equation*}
and since $t_{g,k}$ is a multiple of $s_{g,1}, \dots,s_{g,k-1}$, this equation is only fulfilled if~$s_{g,1} \cdots s_{g,k-1}$ divides $B + 1$. Comparing this to the possible values of $B$, we obtain $B = s_{g,1} \cdots s_{g,k-1} - 1$. Plugging this in for $B$ and solving for $\Lambda_k$, we obtain~$\Lambda_k = s_{g,k} - g$. This shows that $P$ is the second matrix from Theorem \ref{thm:fwps-max-mult}. Finally assume $g$ is odd. We plug in the values for $a_{1d}, \dots,a_{(d-1)d}$ provided in Theorem \ref{thm:fwps-max-mult} (iii) and check that the resulting matrix has primitive columns. This shows that this is a valid choice for $a_{1d}, \dots,a_{(d-1)d}$, which completes the proof.
\end{proof}

\goodbreak

\section{A classification procedure for IP lattice simplices}\label{sec:fwps-classification}

Throughout this section we develop a procedure for the classification of all IP lattice simplices of given dimension and Gorenstein index, see Algorithm~\ref{alg:fullclass}. It is easily adapted to only classify Fano simplices, see Remark \ref{rem:fano-mod}.

\begin{proposition}\label{prop:simp-Pmat-constr}
Fix an integer $d \ge 2$ and $d+2$ positive integers $g, g_0, \dots, g_d$ with~$g~=~\lcm(g_0, \dots, g_d)$. Let $A = (\alpha_0,\dots,\alpha_d) \in \ZZ^{d+1}_{\ge 1}$ a unit fraction partition of $g$, ie.
\begin{equation*}
    \frac{1}{g} \ = \ \frac{1}{\alpha_0} + \dots + \frac{1}{\alpha_d}.
\end{equation*}
Denote by $w = (w_0,\dots,w_d) = Q(A)$ the weight system associated with $A$. Consider the $d \times (d+1)$ integer matrices of the form
\begin{equation*}
    P
    \ := \
    [\, v_0 \ \dots \ v_d\, ]
    \ := \ 
    \left[\begin{array}{ccccc}
        a_{11} & a_{12} & \cdots & a_{1d} & -b_1\\
        0 & a_{22} & \cdots & a_{2d} & -b_2 \\
        \vdots & \ddots & \ddots & \vdots & \vdots \\
        0 & \cdots & 0 & a_{dd} & -b_d \\
    \end{array}\right]
\end{equation*}
such that for all $k = 1, \dots, d$ the entries of $P$ satisfy
\begin{enumerate}
    \item $a_{kk} \in \ZZ_{\ge 1},\ a_{kk} \mid \alpha_{k-1}$,
    \item $0 \le a_{ik} < a_{kk}$ for all $1 \le i < k$,
    \item $b_k w_d = a_{kk} w_{k-1} + \dots + a_{kd}w_{d-1}$.
\end{enumerate}
Let $\Delta := \Delta(P)$ the convex hull of the columns of $P$. Then $\Delta$ is a $d$-dimensional IP lattice simplex whose associated weight system satisfies $Q_\Delta^\red = (w_0,\dots,w_d)$. The~$k$-th Gorenstein form $u_k = (u_{k1},\dots,u_{kd})$ of $\Delta$ is explicitly given by
\begin{equation}\label{eq:gorform}
    u_{kj} \ = \
    \left\{\begin{array}{ll}
        \frac{|w| - w_k}{a_{jj}w_k} - \frac{\sum_{l=1}^{j-1}a_{lj}u_{kl}}{a_{jj}}, & \text{ if } j = k+1, \\
        \frac{- 1 - \sum_{l=1}^{j-1}a_{lj}u_{kl}}{a_{jj}}, & \text{ otherwise.}
    \end{array}\right.
\end{equation}
If each of $g_k u_k$ is a primitive vector in $\ZZ^d$, then $\Delta$ is of Gorenstein index $g$ with local Gorenstein indices $g_k$, where $k = 0,\dots,d$.
\end{proposition}

\begin{proof}
As $P$ is of rank $d$, the polytope $\Delta$ is full-dimensional. Its vertices are precisely the columns $v_0,\dots,v_d$. It is thus a lattice simplex. By condition (iii) we have~$P\cdot w = 0$. With $\beta_k := w_k/|w|$ we can write
\begin{equation*}
\mathbf{0} \ = \ \beta_0 v_0 + \dots + \beta_d v_d,
\end{equation*}
which is a convex combination of $v_0,\dots,v_d$ with non-vanishing coefficients. Thus the origin is contained in the interior of $\Delta$, making it an IP lattice simplex. By Remark \ref{rem:ws-props} (ii) we have $Q_\Delta^\red = w$. Let $u_k = (u_{k1},\dots,u_{kd})$ the $k$-th Gorenstein form of~$\Delta$. Let $1 \le j \le d$. If $j \ne k+1$, then $\bangle{u_k,v_j} = -1$ holds and for $j = k+1$ we have $\bangle{u_k,v_j} = |w|/w_k - 1$. Solving these equations for $u_{kj}$ produces the identities in Equation \ref{eq:gorform}. The last assertion is just the definition of the Gorenstein index and the local Gorenstein indices of $\Delta$.
\end{proof}

\begin{proposition}\label{prop:simp-Pmat-finite}
Let $\Delta$ a $d$-dimensional IP lattice simplex of Gorenstein index~$g$. Then $\Delta \cong \Delta(P)$ holds with a matrix $P$ as provided by Proposition \ref{prop:simp-Pmat-constr}.
\end{proposition}

\begin{proof}
Write $w := Q_\Delta^\red = (w_0,\dots,w_d)$ and let $A = A(\Delta) = (\alpha_0,\dots,\alpha_d)$ the unit fraction partition of $g$ associated with $\Delta$. By Proposition \ref{prop:ws-ufp} the reduced weight systems $w$ and $Q(A(\Delta))$ coincide. Let~$P$ the~$d \times (d+1)$ integer matrix whose columns are the vertices of $\Delta$. By bringing~$P$ in Hermite normal form, we may assume
\begin{equation*}
    P
    \ = \
    [\, v_0 \ \dots \ v_d\, ]
    \ = \
    \left[\begin{array}{ccccc}
        a_{11} & a_{12} & \cdots & a_{1d} & -b_1\\
        0 & a_{22} & \cdots & a_{2d} & -b_2 \\
        \vdots & \ddots & \ddots & \vdots & \vdots \\
        0 & \cdots & 0 & a_{dd} & -b_d \\
    \end{array}\right],
\end{equation*}
where $a_{kk} \in \ZZ_{\ge 1}$ holds for all $k = 1, \dots, d$ as well as $0 \le a_{ik} < a_{kk}$ for all $1 \le i < k$. Resolving $P \cdot w = 0$ for the entries $b_k$, we obtain the identity
\begin{equation*}
    b_{k} w_d \ = \ a_{kk} w_{k-1} + \dots + a_{kd}w_{d-1}
\end{equation*}
It thus remains to show that for all $k$ the diagonal entry $a_{kk}$ divides~$\alpha_{k-1}$. Consider the following sequence of rational numbers
\begin{equation*}
    q_1 \ := \ -\frac{1}{a_{11}}, \qquad q_j \ := \ -\frac{1 + a_{1j}q_1 + \dots + a_{j-1,j}q_{j-1}}{a_{jj}}.
\end{equation*}
Let $k \ge 1$ and let $u_k$ the $k$-th Gorenstein form of $\Delta$. For each $1 \le j \le k$ we have~$\bangle{u_k,v_j} = -1$. Solving this for $u_{kj}$ we get $u_{kj} = q_j$. In particular $g\, q_k$ is an integer. Evaluating $u_{k-1}$ on $v_k$, we obtain
\begin{equation*}
    a_{1k} q_1 + \dots + a_{(k-1)k}q_{k-1} + a_{kk} u_{(k-1)k} \ = \ \bangle{u_{k-1},v_k} \ = \ \frac{|w|}{w_{k-1}} - 1.
\end{equation*}
With the definition of $q_k$, we can rewrite this equation as
\begin{equation*}
    a_{kk} (u_{(k-1)k} - q_k) \ = \ -\frac{|w|}{w_{k-1}}.
\end{equation*}
Note that the $g$-fold of both $u_{(k-1)k}$ and $q_k$ is an integer. Multiplying both sides by~$g\,w_{k-1}$ thus shows that $a_{kk} w_{k-1}$ is a divisor of $g|w| = \alpha_{k-1} w_{k-1}$. Clearing~$w_{k-1}$ on both sides, we see that $a_{kk}$ divides $\alpha_{k-1}$.
\end{proof}

Propositions \ref{prop:simp-Pmat-constr} and \ref{prop:simp-Pmat-finite} provide us with a procedure to enumerate up to isomorphy all IP lattice simplices $\Delta$ with a given constellation of local Gorenstein indices~$(g_0,\dots,g_d)$ and given reduced weight system $w$. The list produced may contain redundancies, ie. matrices $P$ and $P'$ that give isomorphic simplices $\Delta(P)$ and $\Delta(P')$. In practice, we want the list to be redundancy free without having to check each pair of matrices for isomorphy. The solution to this problem is to define a normal form $\NF(P)$ for these matrices $P$, which has the property that two matrices $P$ and $P'$ give isomorphic simplices if and only if their normal forms coincide. The normal form we present in Definition~\ref{def:normal-form} is similar to the PALP normal form for lattice polytopes described in \cite{KrSk04}, see also~\cite{GrKa13}. To fix some notation, if~$B$ is a~$m\times n$ integer matrix with columns $b_1,\dots,b_n$ and $\sigma \in S_n$, then we denote by~$B_\sigma$ the matrix with columns $b_{\sigma(1)},\dots,b_{\sigma(n)}$. Moreover, by $\HNF(B)$ we denote the hermite normal form of $B$.

\begin{definition}\label{def:normal-form}
Let $P$ a $d \times (d+1)$ integer matrix whose columns generate $\QQ^d$ as a convex cone. Let $w = (w_0,\dots,w_d)$ the reduced weight system and $(g_0,\dots,g_d)$ the local Gorenstein indices of the IP lattice simplex $\Delta(P)$. We denote by $S_P$ the subset of $S_{d+1}$ consisting of all permutations $\sigma \in S_{d+1}$ with the following properties:
\begin{enumerate}
    \item
    If $\sigma(i) \le \sigma(j)$ holds, then $w_{\sigma(i)} \ge w_{\sigma(j)}$.
    \item
    If $\sigma(i) \le \sigma(j)$ and $w_{\sigma(i)} = w_{\sigma(j)}$ holds, then $g_{\sigma(i)} \ge g_{\sigma(j)}$.
\end{enumerate}
We define the \emph{normal form} of $P$ as
\begin{equation*}
    \NF(P) \ := \ \min\{ \HNF(P_\sigma); \ \sigma \in S_P \},
\end{equation*}
where the minimum is taken lexicographically, ie. we write the entries of the matrix~$\HNF(P_\sigma) = (h_{ij})_{ij}$ as a list of integers $(h_{11},\dots,h_{1d},h_{21},\dots,h_{(d+1)d})$ and take the lexicographic minimum among those lists.
\end{definition}

\begin{proposition}
For $d \times (d+1)$ integer matrices $P$ and $P'$, whose columns generate $\QQ^d$ as a convex cone, we have $\Delta(P) \cong \Delta(P')$ if and only if their normal forms $\NF(P)$ and $\NF(P')$ coincide.
\end{proposition}

\begin{proof}
Assume $\Delta(P) \cong \Delta(P')$ holds. Then there is a permutation $\sigma \in S_{d+1}$ and a $d \times d$ unimodular matrix $S$ such that $S \cdot P_\sigma = P'$ holds. A quick comparison shows~$S_P = \sigma S_{P'}$. Thus the sets of hermite normal forms, among which the lexicographic minimum is chosen, coincide. We obtain $\NF(P) = \NF(P')$. On the other hand, if~$\NF(P) = \NF(P')$ holds, then there are $\sigma, \sigma' \in S_{d+1}$ and unimodular~$d\times d$ matrices~$S$ and $S'$ with $S\cdot P_\sigma = S' \cdot P'_{\sigma'}$. Thus $\Delta(P)$ and $\Delta(P')$ are isomorphic.
\end{proof}

We translate Proposition \ref{prop:simp-Pmat-constr} into a classification procedure realized in Algorithm~\ref{alg:subclass}. As input it takes a unit fraction partition $A = (\alpha_0,\dots,\alpha_d)$ of $g$ and a tuple $(g_0,\dots,g_d)$ of positive integers with $g = \lcm(g_0,\dots,g_d)$. It then produces a list of matrices $P$ corresponding to IP lattice simplices $\Delta(P)$ with associated unit fraction partition $A$ and local Gorenstein indices $(g_0,\dots,g_d)$. This list is complete, ie. every IP lattice simplex $\Delta$ with associated unit fraction partition $A$ and local Gorenstein indices $(g_0,\dots,g_d)$ is isomorphic to some $\Delta(P)$ with $P$ from that list, and the list is redundancy free, ie. two different matrices $P$ and $P'$ from the list give non-isomorphic simplices $\Delta(P)$ and $\Delta(P')$.

\goodbreak

\begin{algrthm}\label{alg:subclass}
ClassifySimp( $A$, $[g_0,\dots,g_d]$ )
\hfill\hrule
\begin{algorithmic}
\Require  -- A unit fraction partition $A = [\alpha_0,\dots,\alpha_d] \in \ZZ^{n+1}_{\ge 1}$ of $g$,
\Statex  -- A list of positive integers $[g_0,\dots,g_d] \in \ZZ^{n+1}_{\ge 1}$ with $g = \lcm(g_0,\dots,g_d)$
\end{algorithmic}
\hrule
\begin{algorithmic}[1]
\State $L \gets [\ ]$
\State $w \gets Q(A)$
\State $\text{div}_k \gets \{ \text{Divisors of } \alpha_k\}$ for $k = 0,\dots,d-1$
\ForAll{ $(a_{11},a_{12},a_{22},\dots,a_{1d},\dots,a_{dd})$ with $a_{kk} \in \text{div}_{k-1}$ and $ 0 \le a_{ik} < a_{kk}$}
    \State $b_k \gets (a_{kk} w_{k-1} + \dots + a_{kd} w_{d-1}) / w_d$ for $k = 1,\dots,d$
    \State
    $
        P
        \gets
        \left[\begin{array}{ccccc}
            a_{11} & a_{12} & \cdots & a_{1d} & -b_1\\
            0 & a_{22} & \cdots & a_{2d} & -b_2 \\
            \vdots & \ddots & \ddots & \vdots & \vdots \\
            0 & \cdots & 0 & a_{dd} & -b_d \\
        \end{array}\right]
    $
    \State $u_k \gets k$-th linear form as in Equation \ref{eq:gorform} for $k = 0,\dots,d$
    \If{$b_k \in \ZZ$ \textbf{and} $g_k u_k$ is a primitive point in $\ZZ^d$ \textbf{and} $\NF(P) \not\in L$}
        \State add $\NF(P)$ to $L$
    \EndIf
\EndFor
\State \Return $L$
\end{algorithmic}
\end{algrthm}

With Algorithm \ref{alg:subclass} we can classify the $d$-dimensional IP lattice simplices with a fixed constellation of local Gorenstein indices and fixed unit fraction partition. To obtain the classification of all $d$-dimensional IP lattice simplices of Gorenstein index $g$, we thus need a list of all length $d+1$ unit fraction partitions of $g$. This is done by the following Algorithm, which takes as input a reduced positive rational number~$p/q$ and a natural number $n \ge 1$ and produces a list of all ordered unit fraction partitions $\alpha_1 \le \dots \le \alpha_n$ of $p/q$ of length $n$. For two unit fraction partitions~$A$ and~$A'$ we write $A \sim A'$, if they only differ by order.

\begin{algrthm}\label{alg:ufp}
UFP( $p/q$, $n$ )
\hfill\hrule
\begin{algorithmic}
\Require  -- A reduced positive rational $p/q \in \QQ_{>0}$
\Statex -- A positive integer $n \in \ZZ_{\ge 1}$
\end{algorithmic}
\hrule
\begin{algorithmic}[1]
\If{ $n = 1$ \textbf{and} $p = 1$ }
    \State \Return $[(p/q)]$
\EndIf
\State $L \gets [\ ]$
\For{$k = \lceil q/p \rceil, \dots, \lfloor nq/p \rfloor$}
    \State $L_2 \gets \text{UFP}(p/q-1/k, n-1)$
    \ForAll{$(1/\alpha_2,\dots,1/\alpha_n) \in L_2$}
        \If{$(1/k, 1/\alpha_1,\dots,1/a_n) \not\sim A'$ for all $A' \in L$}
            \State sort $(1/k, 1/\alpha_1,\dots,1/a_n)$ decreasingly and add it to $L$
        \EndIf
    \EndFor
\EndFor
\State \Return $L$
\end{algorithmic}
\end{algrthm}

The following Algorithm takes as input integers $d\ge 2$ and $g \ge 1$ and performs the classification of all $d$-dimensional IP lattice simplices of Gorenstein index $g$. As in the case of Algorithm \ref{alg:subclass}, the output list of matrices $P$ is complete and redundancy-free.

\begin{algrthm}\label{alg:fullclass} ClassifyAllSimp( $d$, $g$ )
\hfill\hrule
\begin{algorithmic}
\Require -- An integer $d \ge 2$
\Statex -- An integer $g \ge 1$
\end{algorithmic}
\hrule
\begin{algorithmic}[1]
\State $L \gets [\ ]$
\ForAll{$ A \in \text{UFP}(1/g, d+1)$}
    \State $L_2 \gets [\ ]$
    \ForAll{$(g_0,\dots,g_d)$ with $g_k \mid g$ such that $g = \lcm(g_0,\dots,g_d)$}
        \ForAll{$P \in \text{ClassifySimp}( A, (g_0,\dots,g_d))$}
            \If{$P \not\in L_2$}
                \State add $P$ to $L_2$
            \EndIf
        \EndFor
    \EndFor
    \State Append $L_2$ to $L$
\EndFor
\State \Return $L$
\end{algorithmic}
\end{algrthm}

\begin{remark}\label{rem:fano-mod}
To classify only Fano simplices of given dimension $d$ and Gorenstein index $g$, we perform the following two modifications:
\begin{enumerate}
    \item
    In Algorithm \ref{alg:subclass} line $6$ we consider those matrices $P$ whose columns are all primitive vectors in $\ZZ^d$.

    \item
    In Algorithm \ref{alg:fullclass} line $2$ we only loop over well-formed unit fraction partitions of $g$, see Remark \ref{rem:ws-props} (ii) and Proposition \ref{prop:ws-ufp}.
\end{enumerate}

\end{remark}

\goodbreak

\end{document}